\documentclass{article}
\usepackage{arxiv}
\usepackage{amsmath}
\usepackage{hyperref}       
\usepackage{url}            
\usepackage{booktabs}       
\usepackage{amsfonts}       
\usepackage{nicefrac}       
\usepackage{microtype}      
\usepackage{doi}
\usepackage{graphicx} 
\usepackage{pgfplots}
\pgfplotsset{compat=1.18} 
\usepackage{graphicx} 
\usepackage{float}

\usepackage{bm}
\usepackage{amsmath}
\usepackage{amssymb}
\usepackage{amsthm}
\usepackage{caption}
\usepackage{subcaption}
\usepackage{algorithm}
\usepackage{algpseudocode}  
\usepackage{comment}
\usepackage{mathpazo} 
\usepackage{circuitikz}
\usepackage{tikz}
\usepackage{cleveref}       
\usepackage{xspace}
\usepackage{adjustbox}
\usepackage{mathtools}
\usetikzlibrary{patterns,hobby,decorations.pathmorphing}
\usepackage{cancel}
\usepackage{xcolor}
\usepackage{ulem}  
\usepackage{fourier}
\usepackage[T1]{fontenc}

\newtheorem{theorem}{Theorem}[section]
\newtheorem{lemma}[theorem]{Lemma}
\newtheorem{remark}[theorem]{Remark}

\newcommand{\argmin}{ \mathop{ \mathrm{argmin} } }

\newcommand{\setbulletpoint}[1]{}
 
\newcommand{\mmDelete}[1]{}

\newcommand{\HMnotshow}[1]{}

\newcommand{\newbmJ}{\widetilde{\bm{J}}}
\newcommand{\newbmR}{\widetilde{\bm{R}}}
\newcommand{\newbmQ}{\widetilde{\bm{Q}}}
\newcommand{\newbmB}{\widetilde{\bm{B}}}

\newcommand{\bmA}{\bm{A}}

\newcommand{\bmf}{\bm{f}}

\newcommand{\bmK}{\bm{K}}
\newcommand{\bmv}{\bm{v}}

\newcommand{\bmU}{\bm{U}}
\newcommand{\bmV}{\bm{V}}
\newcommand{\bmw}{\bm{w}}
\newcommand{\bmW}{\bm{W}}
\newcommand{\bmJ}{\bm{J}}
\newcommand{\bmR}{\bm{R}}

\newcommand{\bmB}{\bm{B}}
\newcommand{\bmD}{\bm{D}}

\newcommand{\bmM}{\bm{M}}

\newcommand{\bmQ}{\bm{Q}}
\newcommand{\bmX}{\bm{X}}

\newcommand{\R}{\mathbb{R}}

\newcommand{\Jmat}[1]{\mathbf{J}_{#1}}
\newcommand{\jtwo}{\Jmat{2}}

\newcommand{\curlyH}{\mathcal{H}}
\newcommand{\curlyHr}{\check{\mathcal{H}}}
\newcommand{\DivHamFOM}{\nabla_{\statex}\curlyH(\statex(t))}

\newcommand{\statex}{{\bm x}}
\newcommand{\outputy}{{\bm y}}
\newcommand{\stateapprox}{\tilde{\bm x}}
\newcommand{\stateRed}{\check{\bm x}}
\newcommand{\outputRed}{\check{\bm y}}
\newcommand{\statexT}{\statex(t)}
\newcommand{\outputyT}{\outputy(t)}
\newcommand{\stateRedT}{\stateRed(t)}
\newcommand{\outputRedT}{\outputRed(t)}
\newcommand{\inputu}{\bm{u}}

\newcommand{\ddt}{\frac{\rm d}{{\rm d} t}}

\newcommand{\bzero}{\mathbf{0}}

\newcommand{\SysLable}{\Sigma}

\newcommand{\snapsize}{n_t}

\newcommand{\portsize}{m}
\newcommand{\FOMsize}{N}
\newcommand{\ROMsize}{r}

\newcommand{\nlsize}{r_n}

\newcommand{\FOMsizeHalf}{n}

\newcommand{\Vbar}{\overline{\bmV}}

\newcommand{\approxfunc}{\varphi}
\newcommand{\decfunc}{\eta}

\newcommand{\matset}{\mathbf{S}}
\newcommand{\matfunc}{\mathcal{F}}

\newcommand{\funcV}{\mathcal{V}}
\newcommand{\funcW}{\mathcal{W}}
\newcommand{\funcJ}{\bm{J}}
\newcommand{\funcR}{\bm{R}}

\newcommand{\funcB}{\bm{B}}

\newcommand{\funcJr}{\check{\mathcal{J}}}
\newcommand{\funcRr}{\check{\mathcal{R}}}

\newcommand{\funcHr}{\check{\mathcal{H}}}
\newcommand{\funcBr}{\check{\mathcal{B}}}
\newcommand{\residual}{{\bm r}}

\newcommand{\JacRed}{\bmD_{\stateRed}\approxfunc(\stateRed)}

\newcommand{\ConFuncSet}{\bm{C}}

\newcommand{\VPOD}{\bmV_{\rm POD}}

\newcommand{\VPODGMG}{\overline{\bmV}}
\newcommand{\VPODQMone}{\overline{\bmV}_{1}}
\newcommand{\VPODQMtwo}{\overline{\bmV}_{2}}

\newcommand{\DEIMsplitHam}{p}

\newcommand{\Jcheck}{\check{\bmJ}}
\newcommand{\Rcheck}{\check{\bmR}}



\newcommand{\pH}{\textsf{pH}\xspace}
\newcommand{\GMG}{\textsf{GMG}\xspace}
\newcommand{\QM}{\textsf{QM}\xspace} 
\newcommand{\DEIM}{\textsf{DEIM}\xspace}
\newcommand{\DMD}{\textsf{DMD}\xspace}
\newcommand{\pHIN}{\textsf{pHIN}\xspace}
\newcommand{\OpInf}{\textsf{OpInf}\xspace}
\newcommand{\IRKA}{\textsf{IRKA}\xspace}
\newcommand{\POD}{\textsf{POD}\xspace}
\newcommand{\FOM}{\textsf{FOM}\xspace}
\newcommand{\ROM}{\textsf{ROM}\xspace}

\newcommand{\ROMs}{\textsf{ROM}s\xspace}
\newcommand{\MOR}{\textsf{MOR}\xspace}

\newcommand{\SPQuad}{\textsf{SP2}\xspace}
\newcommand{\SPLin}{\textsf{SP1}\xspace}

\newcommand{\framework}{framework \xspace}

\newcommand\tstrut{\rule{0pt}{2.8ex}}
\newcommand\AxisLineWidth{0.5pt}

\newcommand\LineWidth{0.8pt}
\newcommand\MarkerSize{1.8pt}

\newcommand{\MarkerSP}{triangle}
\newcommand{\MarkerSPG}{triangle}
\newcommand{\MarkerGMG}{triangle}
\newcommand{\MarkerGMGQM}{triangle}
\newcommand{\MarkerProj}{square}

\newcommand{\LineStyleSP}{dashed}
\newcommand{\LineStyleSPG}{dashed}
\newcommand{\LineStyleGMG}{solid}
\newcommand{\LineStyleProj}{dotted}

\newcommand{\ColorFOM}{black}
\newcommand{\ColorSP}{red}
\newcommand{\ColorSPG}{red}
\newcommand{\ColorGMG}{teal}
\newcommand{\ColorGMGQM}{blue}
\newcommand{\ColorLowbound}{cyan}
\renewcommand{\emph}[1]{\textit{#1}}
\newcommand*\samethanks[1][\value{footnote}]{\footnotemark[#1]}

\algrenewcommand\algorithmicrequire{\textbf{Input:}}
\algrenewcommand\algorithmicensure{\textbf{Output:}}

\title{Structure-Preserving Generalized Manifold Galerkin Reduction for Port-Hamiltonian Systems}

\date{}

\author{ 
	{
		\Large Silke ~Glas \thanks{
			Department of Applied Mathematics
			University of Twente
			Enschede, The Netherlands
			(\texttt{s.m.glas@utwente.nl}, \texttt{h.l.mu@utwente.nl})}}
	\And
	{
		\Large Hongliang~Mu \samethanks[1]} 
}

\renewcommand{\headeright}{}
\renewcommand{\undertitle}{}
\renewcommand{\shorttitle}{Structure-Preserving GMG Reduction for Port-Hamiltonian Systems}

\DeclareMathAlphabet{\mathcal}{OMS}{cmsy}{m}{n}

\begin{document}
	\maketitle
	
\begin{abstract}
		This paper considers structure-preserving model order reduction (\MOR) techniques for port-Hamiltonian (\pH) systems, which are typically derived from energy-based modeling.
		To keep favorable properties of \pH systems such as passivity in a reduced order model (\ROM), we use structure-preserving methods in the reduction process. 
		Although projection-based structure-preserving \MOR methods for nonlinear pH systems based on nonlinear approximation ansatzes have recently been proposed, existing approaches typically rely on specific structures of the approximation map and the underlying \pH system. 
		To address this limitation, we propose a \MOR framework based on generalized manifold Galerkin (\GMG) reduction. The resulting framework can employ general nonlinear approximation maps while preserving the \pH structure. We establish sufficient conditions for structure preservation, show that the associated non-degeneracy conditions are generically satisfied. 
		We further present linear and quadratic approximation maps within the proposed framework.
		Numerical examples for a linear and a nonlinear mass-spring-damper system show that the proposed \MOR methods have lower relative reduction error compared to existing methods. 
\end{abstract}
	

\section{Introduction}
		Port-Hamiltonian (\pH) systems are often derived via energy-based modeling and arise in a wide range of applications such as electrical circuits \cite{falaize2016passive}, mechanical systems \cite{duindam2009modeling}, and fluid dynamics \cite{califano2021geometric}. By using this type of modeling, a large-scale system can be decomposed into subsystems which are interconnected through their ports. 
		An important property of these \pH systems is that they satisfy an energy balance equation, which states that
		the rate of change of the Hamiltonian equals the supplied energy minus the dissipated energy. This relation can be used to derive system properties such as passivity. 
		However, the simulation of high-dimensional \pH systems can be computationally expensive. 
		To overcome this bottleneck, model order reduction (\MOR) techniques have been developed that approximate the full-order model (\FOM) by a lower-dimensional reduced-order model (\ROM), see e.g., \cite{Antoulas2005, BennerGugercinWilcox2015, BennerOhlbergerCW2017, Hesthaven_Peherstorfer_Unger_2026} for an overview. Well-known \MOR methods include the proper orthogonal decomposition (\POD)\cite{sirovich1987turbulence}, the iterative rational Krylov algorithm (\IRKA) \cite{antoulas2010interpolatory}, and the balanced truncation \cite{moore1981principal}. While these methods have been successfully applied in a wide range of settings, they do not generally preserve the \pH structure of the original system in the \ROM.
		However, numerous structure-preserving \MOR methods for \pH systems have been developed in recent years. In the following, we provide an overview of these methods.
		
		We first review MOR methods for linear pH systems. Structure-preserving \MOR approaches based on moment matching were introduced in \cite{ionescu2013families, wolf2010passivity}.
		Based on this framework, \pH-\IRKA was proposed in \cite{gugercin2012structure}.  
		Extensions of interpolatory structure-preserving \MOR to \pH differential-algebraic systems were considered in \cite{Beattie2022}.
		A Riemannian optimization approach for \MOR of \pH systems based on a trust region methods was developed in \cite{sato2018riemannian}. Another optimization-based approach named structured optimization-based model order reduction was introduced in \cite{schwerdtner2023sobmor} and constructs \ROMs by solving a parameter optimization problem. 
		Furthermore, non-intrusive structure-preserving \MOR methods, including  \pH-dynamic mode decomposition (\pH\DMD) \cite{morandin2023port}, \pH identification network (\pHIN) \cite{rettberg2024data}, and data-driven \pH-operator inference (\pH-\OpInf) \cite{geng2025data}, have been proposed in the last years.
		For \pH systems with nonquadratic Hamiltonian, multiple \MOR approaches have been introduced. 
		In \cite{ionescu2013moment}, a moment matching-based MOR method for nonlinear pH systems was developed. A structure-preserving balanced truncation \MOR method based on generalized Gramians was proposed in \cite{kawano2018structure}.
		Moreover, a \POD-based structure-preserving \MOR technique was developed by considering a \POD basis of the gradient of the Hamiltonian $\nabla_{\statex}\curlyH(\statex)$ \cite{chaturantabut2016structure}. For \pH systems with nonquadratic Hamiltonian, it is important to note that the evaluation of the \ROM still depends on the dimension of the \FOM even if a linear reduction technique is utilized. Therefore, structure-preserving discrete empirical interpolation methods (\DEIM) proposed in \cite{chaturantabut2016structure, pagliantini2023gradient}, are typically applied to reduce computational complexity.
		When the underlying system is advection-dominated, its approximation quality using linear \MOR methods might suffer from slowly decaying Kolmogorov $n$-widths, see e.g., \cite{greif2019decay, ohlberger2016reduced}. 
		To overcome this limitation, nonlinear \MOR methods can be used. 
		For \pH systems, the work of \cite{schulze2023structure} recently introduced a nonlinear \MOR method based on a special class of nonlinear approximation maps. 

		In this work, we propose a structure-preserving \MOR method for \pH systems based on the generalized manifold Galerkin (\GMG) reduction \cite{buchfink2024model}. The main contributions of this work are
\begin{enumerate}
	\item[(i)] We derive sufficient conditions under which the \GMG reduction preserves the \pH structure in Theorem \ref{Theo:pH_structure}. The resulting method is applicable to nonlinear \pH systems and allows for use of general nonlinear approximation maps. 
	\item [(ii)] We derive existence results for approximation maps in Section \ref{Sec:Existence_proofs} by showing that the associated non-degeneracy conditions are generically satisfied. 
	\item[(iii)] We present linear (Section \ref{Sec:GMGPODROM}) and quadratic (Section \ref{Sec:GMGQMROM}) realizations of the proposed \MOR method, resulting in the reduced models  \GMG-\POD-\ROM and \GMG-\QM-\ROM. We demonstrate on a linear and a nonlinear mass-spring damper system that these \ROMs achieve lower approximation error than existing methods. 
\end{enumerate}

The outline of this paper is as follows: in Section \ref{Sec:Prelim}, we introduce \pH systems and the \GMG reduction. We propose the novel \MOR technique that preserves the \pH structure and is based on the \GMG reduction in Section \ref{Sec:GMG_reduction}.
	In Section \ref{Sec:GMG_state_approx}, we establish generic existence results for approximations maps and detail how the \ROM is derived for a linear and a quadratic embedding functions. Numerical results for a linear and a nonlinear mass-spring-damper system are discussed in Section \ref{Sec:Numerical_results}. 
	We conclude and discuss possible extensions of this work in Section \ref{Sec:Outlook}. 


	\section{Preliminaries} \label{Sec:Prelim}
	
	In this section, we introduce the concepts and methods needed in the remainder of the paper. Section \ref{subsec:pH_intro} recalls \pH systems and their properties. 
	Next, we review the generalized manifold Galerkin (\GMG) reduction in Section \ref{subsec:gmg_recall}. 
		
	\subsection{Introduction to port-Hamiltonian systems}\label{subsec:pH_intro}
	
	We consider $\FOMsize$-dimensional \emph{input-state-output port-Hamiltonian (\pH) systems} of the form\footnote{We consider here state-independent system matrices $\funcJ, \funcR$ and $\funcB$ for readability of the paper. Nevertheless, the model reduction method proposed in this paper also applies to state-dependent system matrices.} 
	\begin{align}
		\label{equ:Nonlinear_FOM_pH}
		\SysLable(\funcJ,\funcR,\curlyH,\funcB, \inputu,\outputy,\statex_0)\coloneqq
		\left\{
		\begin{array}{rl}
			\ddt \statexT &=\left(\funcJ-\funcR\right)
			\DivHamFOM
			+\funcB \inputu(t),
			\\
			\outputyT &=\funcB^{\top}\DivHamFOM,
			\quad \statex(t_0) = \statex_0,
		\end{array}
		\right.
	\end{align}
	with state $\statex(t) \in \R^{\FOMsize}$, input $\inputu(t) \in \R^{\portsize}$, output $\outputy(t) \in \R^{\portsize}$ on time interval $\mathcal{I}:=(t_0,t_{\text{end}}]$ and initial condition $\statex_0\in\R^{\FOMsize}$. 
	This \pH system exchanges energy to its environment via ports, and $\funcB \in \R^{\FOMsize\times \portsize}$, also called port-matrix, describes this exchange. 
	The function $ \curlyH \in \mathcal{C}^{1}(\R^{\FOMsize}, \R)$, called the \emph{Hamiltonian}, represents the internal energy of the system. We assume that $\curlyH$ is bounded from below, i.e., $\exists \statex_e\in\R^{\FOMsize}$ such that  
	\begin{align*} 
		\curlyH(\statex) \geq \curlyH(\statex_e), \quad \forall
		\statex \in \R^{\FOMsize}.
	\end{align*}
	Further, we have the matrices $\funcJ,\funcR \in \R^{\FOMsize\times \FOMsize}$ related to the interconnection and dissipation of the system, which satisfy
	\begin{align*}
		\funcJ = -\funcJ^{\top} \in \R^{\FOMsize\times \FOMsize},\quad
		0\preceq \funcR = \funcR^{\top} \in \R^{\FOMsize\times \FOMsize},\quad
		{\rm det}\left(\funcJ-\funcR\right)\neq 0.
	\end{align*} 
	The latter condition, i.e.,  ${\rm det}\left(\funcJ-\funcR\right)\neq 0$, is not part of the standard definition of \pH systems. It is imposed throughout this paper. 
	One important property of \pH systems $\SysLable(\funcJ,\funcR,\curlyH,\funcB, \inputu,\outputy,\statex_0)$ is that they satisfy a power balance equation, i.e., 
	\begin{align}
		\label{equ:energy_conserve}
		\outputyT^{\top}\inputu(t)
		=\ddt \left(\statexT^{\top} \nabla_{\statex}\curlyH\right)+
		\left(\funcR\nabla_{\statex}\curlyH(\statexT)\right)^{\top}\nabla_{\statex}\curlyH(\statexT).
	\end{align}
	The above equation shows that the power supplied through the ports equals the rate of change of stored energy and the dissipated power \cite{polyuga2012effort}. Here, we denote $\outputyT^{\top}\inputu(t)$ and $(\funcR\nabla_{\statex}\curlyH(\statexT))^{\top}\nabla_{\statex}\curlyH(\statexT)$ as power supply rate and power dissipation rate. By integration of \eqref{equ:energy_conserve} over the time period $[0,T]$, we arrive at the energy balance equation for \pH systems
	\begin{align}
		\label{equ:energy_conserve_time_int}
		\curlyH(\statex(T))-\curlyH(\statex(0))
		=
		{\int_{t=0}^{T}\outputyT^{\top}\inputu(t) \ {\rm d}t}
		-
			\int_{t=0}^{T}\left(\funcR\nabla_{\statex}\curlyH(\statexT)\right)^{\top}\nabla_{\statex}\curlyH(\statexT) \ {\rm d}t.
	\end{align}	
	For more details on \pH systems, we refer the reader to, e.g., \cite{jacob2012linear, vanderschaft2014port, rashad2020,mehrmann2023control}.

	
	\subsection{Generalized manifold Galerkin}
	\label{subsec:gmg_recall}
The differential-geometric framework introduced in \cite{buchfink2024model} provides a general setting for model order reduction on manifolds. One reduction method allowing for structure preservation within this framework is the generalized manifold Galerkin (\GMG) reduction. Since the structure-preserving reduction method proposed in this paper is based on the \GMG reduction, we briefly recall the essential ingredients of the framework in the following. Throughout the paper, we work on a fixed chart and use the corresponding coordinate representation of the involved quantities. For a manifold-based formulation of the \GMG reduction, we refer to \cite[Section 5]{buchfink2024model}. 
	
	As our full order model (\FOM), we consider the initial value problem 
	\begin{align*}
		\ddt \statex(t) = \bmf (\statex(t)) \quad \text{in }  \R^{\FOMsize}, t \in \mathcal{I}, \qquad \statex(t_0) = \statex_0 \quad \text{in } \R^{\FOMsize}, 
	\end{align*}
	with $\bmf : \R^{\FOMsize} \rightarrow \R^{\FOMsize}$. Our goal is to derive a \ROM of the form 
	\begin{align}
		\label{eq:ROMManiMOR}
		\ddt \stateRed(t) = \check{\bmf} (\stateRed(t)) \quad \text{in }  \R^{\ROMsize}, t \in \mathcal{I}, \qquad \stateRed(t_0) = \stateRed_0 \quad \text{in } \R^{\ROMsize}, 
	\end{align}
	where $\check{\bmf} : \R^{\ROMsize} \rightarrow \R^{\ROMsize}$ and $\ROMsize \ll \FOMsize$. To this end, we define an embedding map and a point reduction map by 
	\begin{align*}
		\approxfunc :  \R^{\ROMsize} \rightarrow  \R^{\FOMsize}, \qquad  \rho :  \R^{\FOMsize} \rightarrow  \R^{\ROMsize}
	\end{align*} 
	satisfying the point projection property, i.e., $\rho \circ \approxfunc = \mathbb{I}_{\ROMsize}$. As a consequence, $\approxfunc \circ \rho$ defines a projection onto the trial manifold parameterized by $\approxfunc$. 
	To define a projection not only on the manifold itself, but also on its tangent bundle, we use the tangent embedding map and the tangent reduction map
	\begin{align*}
		\bmD_{\stateRed}\approxfunc :  \R^{\ROMsize} \rightarrow  \R^{\FOMsize}, \qquad  \mathcal{R}_{\text{red},\statex} :  \R^{\FOMsize} \rightarrow  \R^{\ROMsize},
	\end{align*}
	where $\bmD_{\stateRed}\approxfunc$ denotes the Jacobian of $\approxfunc$. We assume that the tangent embedding and tangent reduction map satisfy $\mathcal{R}_{\text{red}, \approxfunc(\stateRed)}  \circ \bmD_{\stateRed}\approxfunc= \mathbb{I}_{\ROMsize} $ for all $\stateRed \in \R^{\ROMsize}$. The reduced vector field and the reduced initial condition in \eqref{eq:ROMManiMOR} are defined by
	\begin{align*}
		\check{\bmf}(\stateRed) :=	 \mathcal{R}_{\text{red},\approxfunc(\stateRed)} \left(\bmf(\approxfunc(\stateRed))\right) \quad  \text{in } \R^{\ROMsize}  \qquad\text{and}\qquad \stateRed_0 := \rho(\statex_0).
	\end{align*} 
	With this setting at hand, we can introduce the \GMG reduction, which is a particular kind of tangent reduction map that allows for structure preservation. To this end, let $\bm{G} \in \R^{\FOMsize \times \FOMsize}$ be a non-degenerate matrix encoding the underlying structure of the system. For example, in the reduction of Hamiltonian systems, $\bm{G}$ may be chosen as the canonical Poisson matrix. Then, the \emph{\GMG reduction} is defined by
	\begin{align*}
		\mathcal{R}_{\text{red},\approxfunc(\stateRed)} \left(\bmf(\approxfunc(\stateRed))\right) = 	\left( \left(\bmD_{\stateRed}\approxfunc(\stateRed)\right)^{\top} \bm{G}  \bmD_{\stateRed}\approxfunc(\stateRed) \right)^{-1} \left(\bmD_{\stateRed}\approxfunc(\stateRed)\right)^{\top} \bm{G} \bmf(\approxfunc(\stateRed)).
	\end{align*} 
	For notational convenience, we denote the \GMG reduction map in this paper by 
	\begin{align}
	\label{equ:GMG_redu_map}
		\mathcal{F}_{\bm{G}} \left(\bmD_{\stateRed}\approxfunc(\stateRed) \right) = 	\left( \left(\bmD_{\stateRed}\approxfunc(\stateRed)\right)^{\top} \bm{G}  \bmD_{\stateRed}\approxfunc(\stateRed) \right)^{-1} \left(\bmD_{\stateRed}\approxfunc(\stateRed)\right)^{\top} \bm{G}. 
	\end{align}
	The above definition of the \GMG reduction map assumes that $\left(\bmD_{\stateRed}\approxfunc(\stateRed)\right)^{\top} \bm{G} \bmD_{\stateRed}\approxfunc(\stateRed)$ is non-degenerate. We therefore define $\matset_{\bm{G}}$ as the set of all matrices satisfying this non-degeneracy condition by 
	\begin{align}
		\label{equ:mat_set_defi}
		\matset_{\bm{G}}\coloneqq \left\{ \bmV\in\R^{\FOMsize\times \ROMsize} \mid {\rm det}(\bmV^{\top} \bm{G} \bmV)\neq 0,\  1\leq\ROMsize\leq\FOMsize \right\}.
	\end{align}
	
	Although the \GMG reduction provides a natural mechanism for preserving structure in a \ROM, it does not immediately provide a reduction method that preserves the input-output structure of \pH systems. In Section \ref{Sec:GMG_reduction}, we show how the \GMG reduction can be adapted to obtain structure-preserving \ROMs of \pH systems.

\section{Generalized manifold Galerkin for \pH systems}
\label{Sec:GMG_reduction}	
We now detail how the \GMG reduction can be utilized to obtain structure-preserving \ROMs of \pH systems. 
To this end, let $\SysLable(\funcJ,\funcR,\curlyH,\bmB, \inputu,\outputy,\statex_0)$ be a \pH system of the form \eqref{equ:Nonlinear_FOM_pH}. Our goal is to approximate the full-order state trajectory by %
\begin{align}
\label{equ:nonlinear_approximation}
	\stateapprox(t)=\approxfunc(\stateRed(t)) \approx \statex(t),
\end{align}
where $\stateRed(t) \in\R^{\ROMsize}$ is the reduced state, $\stateapprox(t) \in\R^{\FOMsize}$ is the reconstructed state and $\approxfunc\colon\R^{\ROMsize}\rightarrow\R^{\FOMsize}$ with $\portsize<\ROMsize\ll \FOMsize$ is a (possibly nonlinear) approximation map. 
								
To derive a \ROM we first define the time-continuous residual of the \FOM \eqref{equ:Nonlinear_FOM_pH} with respect to the reconstructed solution $\stateapprox(t)$ by
\begin{align*}
	\residual(t)\coloneqq &\ddt \stateapprox(t)-\left(\funcJ-\funcR\right)\nabla_{\statex}\curlyH(\stateapprox(t)) -\bmB\inputu(t) \\
	=  &\left( \bmD_{\stateRed}\approxfunc(\stateRed) \right) \ddt \stateRedT -\left(\funcJ-\funcR\right)\nabla_{\statex}\curlyH(\approxfunc(\stateRedT)) -\bmB\inputu(t). 
\end{align*} 
Then, we require that by using the \GMG reduction \eqref{equ:GMG_redu_map} the projected residual vanishes, i.e., 
\begin{align*}
	\funcW(\stateRedT)^{\top}\residual(t) \overset{!}{=} \bzero, 
\end{align*}
where we denote by $\funcW\colon\R^{\ROMsize}\rightarrow\R^{\FOMsize\times\ROMsize}$ the \GMG reduction map with 
\begin{align}
	\label{equ:GMG_NL_Construct_W}
	\funcW\left( \stateRed \right) \coloneq \left(\matfunc_{\left(\funcJ-\funcR\right)^{-1}} \left(\bmD_{\stateRed}\approxfunc(\stateRed) \right) \right)^{\top}.
\end{align}
By reformulating the latter equation, we arrive at the following \ROM
\begin{subequations}\label{equ:ROM_non_pH_structure}
\begin{align}
	\ddt \stateRedT &=
	\funcW(\stateRedT)^{\top}\left(\funcJ-\funcR \right)\nabla_{\statex}\curlyH(\approxfunc(\stateRedT)) + \funcW(\stateRedT)^{\top}\bmB\inputu(t),\label{equ:ROM_non_pH_structure_A}\\
	\outputRedT&=\bmB^{\top}\nabla_{\statex}\curlyH(\approxfunc(\stateRedT)), \quad \stateRed(0) = \stateRed_0 , \label{equ:ROM_non_pH_structure_B}
\end{align}
\end{subequations}
where $\outputRedT \in \R^{\portsize}$ denotes the reduced output and $\stateRed_0 \in \R^{r}$ the reduced initial condition. The following theorem characterizes the conditions under which the \ROM \eqref{equ:ROM_non_pH_structure} retains the \pH structure.

\begin{theorem}
\label{Theo:pH_structure}
Let $\SysLable(\funcJ,\funcR,\curlyH,\bmB, \inputu,\outputy,\statex_0)$ be an N-dimensional \pH system of the form \eqref{equ:Nonlinear_FOM_pH}, and let  $\approxfunc\colon\R^{\ROMsize}\mapsto\R^{\FOMsize}$ be an approximation map. Assume that for every $\stateRed\in\R^{\ROMsize}$, ${\rm span}(\bmB)\subseteq {\rm span}(\JacRed)$ and $\bmD_{\stateRed}\approxfunc(\stateRed)\in\matset_{\left(\funcJ-\funcR\right)^{-1}}$, see \eqref{equ:mat_set_defi}. Then, the \ROM obtained by using the \GMG reduction in \eqref{equ:ROM_non_pH_structure} is again a \pH system.
\end{theorem}
\begin{proof}
Let $\funcV(\stateRed)\coloneq\bmD_{\stateRed}\approxfunc(\stateRed)$. By the definition of the \GMG reduction map \eqref{equ:GMG_NL_Construct_W}
we have \begin{align*}\funcW(\stateRed)^{\top} \funcV(\stateRed) = \mathbb{I}_r.\end{align*} 
Using \eqref{equ:ROM_non_pH_structure_A}
\begin{align*}
		\ddt \stateRedT &=
		\funcW(\stateRed)^{\top}\left(\funcJ-\funcR \right)\nabla_{\statex}\curlyH(\approxfunc(\stateRedT)) + \funcW(\stateRed)^{\top}\bmB\inputu(t) \\
		&= \left( \funcV(\stateRed)^{\top}\left(\funcJ-\funcR \right)^{-1} \funcV(\stateRed) \right)^{-1} \funcV^{\top} \nabla_{\statex}\curlyH(\approxfunc(\stateRedT)) + \funcW(\stateRed)^{\top}\bmB\inputu(t). 
\end{align*}
Defining $	\curlyHr(\stateRed) \coloneq \curlyH(\approxfunc(\stateRed))$ and applying the chain rule yields
\begin{align*}
	\ddt \stateRedT =\left( \funcV(\stateRed)^{\top}\left(\funcJ-\funcR \right)^{-1} \funcV(\stateRed) \right)^{-1}  \nabla_{\stateRed}\curlyHr(\stateRedT)  + \funcW(\stateRed)^{\top}\bmB\inputu(t).
\end{align*}
Moreover, by the definition of the \GMG reduction map, we get 
\begin{align*}
	\funcW(\stateRed)^{\top}\left(\funcJ-\funcR \right) \funcW(\stateRed) = \left( \funcV(\stateRed)^{\top}\left(\funcJ-\funcR \right)^{-1} \funcV(\stateRed) \right)^{-1},
\end{align*}
and consequently
\begin{align*}
	\ddt \stateRedT =\funcW(\stateRed)^{\top}\left(\funcJ-\funcR \right) \funcW(\stateRed)  \nabla_{\stateRed}\curlyHr(\stateRedT)  + \funcW(\stateRed)^{\top}\bmB\inputu(t).
\end{align*}
Setting the reduced quantities by 
$
		\funcJr(\stateRed) \coloneqq\funcW(\stateRed)^{\top}\funcJ\funcW(\stateRed), 
		\funcRr(\stateRed) \coloneqq\funcW(\stateRed)^{\top}\funcR\funcW(\stateRed)$,
yields
\begin{align*}
	\ddt \stateRedT =\left( \funcJr(\stateRed) - \funcRr(\stateRed)  \right) \nabla_{\stateRed}\curlyHr(\stateRedT)  + \funcW(\stateRed)^{\top}\bmB\inputu(t).
\end{align*}
Note that $\funcJr(\stateRed) = -\funcJr(\stateRed)^{\top}$, $\funcRr(\stateRed) = \funcRr(\stateRed)^{\top} \succeq 0$ and the reduced Hamiltonian $\curlyHr$ is continuously differentiable. It remains to verify that the output equation of \eqref{equ:ROM_non_pH_structure} can be written in \pH form. 
Since ${\rm span}(\bmB)\subseteq{\rm span}(\funcV(\stateRed))$, there exists a matrix $\bm{M} \in \R^{\ROMsize \times \portsize}$ such that $\bmB =\funcV(\stateRed) \bm{M} $. Further, as $\funcV(\stateRed)^{\top}\funcW(\stateRed) = \mathbb{I}_r$, we get 
\begin{align}
	\label{equ:cond_on_B}
	\bmB^{\top}=  \bm{M}^{\top}  \funcV(\stateRed)^{\top} =  \bm{M}^{\top} \funcV(\stateRed)^{\top}\funcW(\stateRed)  \funcV(\stateRed)^{\top} = \bmB^{\top}\funcW(\stateRed) \funcV(\stateRed)^{\top}.
\end{align}		
Hence, we rewrite \eqref{equ:ROM_non_pH_structure_B} by	
\begin{align*}
	\bmB^{\top}\nabla_{\statex}\curlyH(\approxfunc(\stateRed)) =\bmB^{\top}\funcW(\stateRed) \funcV(\stateRed)^{\top}\nabla_{\statex}\curlyH(\approxfunc(\stateRed)) =\funcBr(\stateRed)^{\top}\nabla_{\stateRed}\curlyHr(\stateRed),
\end{align*}
where $\funcBr(\stateRed)\coloneqq\funcW(\stateRed)^{\top}\bmB \in \R^{\ROMsize \times \portsize}$. As a consequence, \eqref{equ:ROM_non_pH_structure} is an $\ROMsize$-dimensional \pH system which can be described by $\SysLable(\funcJr(\stateRed),\funcRr(\stateRed),\funcHr,\funcBr(\stateRed),\inputu,\outputRed,\stateRed_0)$.
\end{proof}

								
\section{State approximation and corresponding \ROMs}
\label{Sec:GMG_state_approx}

In this section, we discuss how to construct approximation maps $\approxfunc$ \eqref{equ:nonlinear_approximation} whose Jacobians satisfy the assumptions of Theorem \ref{Theo:pH_structure} by construction. We further present specific realizations of such approximation maps and the resulting \ROMs. To this end, we consider approximation maps of the form 
\begin{align}
\label{equ:general_embedding_map}
									\approxfunc(\stateRed) =\bmB \stateRed_1 + \Vbar \decfunc\left(\stateRed_2\right), \quad	\stateRed_1\in\R^{\portsize}, \quad 	\stateRed_2\in\R^{\ROMsize-\portsize}, \quad	\stateRed=\begin{bmatrix}\stateRed_1 \\ \stateRed_2\end{bmatrix},
\end{align}
where $\Vbar\in\R^{\FOMsize\times (\nlsize+\ROMsize-\portsize)}$ satisfies $\bmB^{\top} \Vbar= \bm{0}_{m \times (\nlsize+\ROMsize-\portsize)}$ and $\decfunc\in\ConFuncSet^1(\R^{\ROMsize-\portsize}, \R^{\nlsize+\ROMsize-\portsize})$. The dimension $r_n$, with $0\leq r_n \ll N$, is chosen according to the concrete realization of $\eta$. While $\eta$ may be any continuously differentiable nonlinear map, such as a neural network, we consider linear and quadratic approximation maps in this paper. These choices and the resulting \ROMs are discussed in Section \ref{Sec:state-approx_resulting_ROMs}. 
								
\subsection{Constructing \pH systems with state approximation}
								Theorem \ref{Theo:pH_structure} shows that the \GMG reduction map preserves the \pH structure provided that
\begin{enumerate}
	\item[(i)] ${\rm span}(\bmB)\subseteq {\rm span}(\JacRed)$ and 
	\item[(ii)] $\bmD_{\stateRed}\approxfunc(\stateRed)\in\matset_{\left(\funcJ-\funcR \right)^{-1}}$. 
\end{enumerate}
Due to the chosen structure of the approximation map \eqref{equ:general_embedding_map}, its Jacobian is given by
\begin{align}\label{eq:jac_approxfunc}
\bmD_{\stateRed}\approxfunc(\stateRed) = 
								\begin{bmatrix}
									\bmB,& \Vbar\bmD_{\stateRed_2}\decfunc(\stateRed_2)
								\end{bmatrix}\in \R^{\FOMsize \times \ROMsize}
\end{align}								
such that condition (i) is satisfied by construction. Moreover, $\funcBr$ is constant and given by 
\begin{align*}
\funcBr(\stateRed)\coloneq
								\begin{bmatrix}
									\mathbb{I}_{\portsize}\\
									\bzero_{\ROMsize-\portsize,\portsize}
								\end{bmatrix}, 
\end{align*}								
which follows from \eqref{equ:cond_on_B} together with ${\rm span}(\bmB)\subseteq {\rm span}(\JacRed)$.

\subsubsection{Existence of admissible approximation maps} \label{Sec:Existence_proofs}
As shown above, condition (i) of Theorem \ref{Theo:pH_structure} is satisfied by the construction of \eqref{equ:general_embedding_map}. Therefore, it remains to verify condition (ii). In this section, we show that the non-degeneracy condition $\bmD_{\stateRed}\approxfunc(\stateRed)\in\matset_{\left(\funcJ-\funcR \right)^{-1}}$ is generically satisfied. This condition is required to ensure that the \GMG reduction is well-defined. Since it is not immediately clear whether matrices satisfying this condition exist, we establish generic existence results in the following.
More precisely, we first show that existence results hold for almost all matrices $\bmU \in \R^{ \FOMsize \times \ROMsize}$. . 
Subsequently, we consider the case of $[\bmB,\bmU] \in \R^{ \FOMsize \times \ROMsize}$ satisfying $[\bmB,\bmU] \in\matset_{\left(\funcJ-\funcR \right)^{-1}}$ as this is the structure we observe in the Jacobian of $\approxfunc$, see \eqref{eq:jac_approxfunc}. For approximation maps with state-dependent Jacobians, this condition is then required to hold pointwise. The following lemma is used repeatedly in the existence proofs. 
 \begin{lemma}
		\label{lemma:analytic_measure}
		Let $f:\R^{\FOMsize}\rightarrow\R$ be a real analytic function and let $m_{\FOMsize}$ be the Lebesgue measure on $\R^{\FOMsize}$. If $m_{\FOMsize}(\{x\in\R^{\FOMsize}| f(x)=0\})>0$, then $f\equiv0$.
	\end{lemma}
\begin{proof}
	Proposition 1 in \cite{Mityagin2020} states that if $f \not\equiv 0$, then $m_{\FOMsize}(\{x\in\R^{\FOMsize}| f(x)=0\})=0$. Taking the contrapositive yields the claim.
\end{proof}	
Throughout this section, let $\bmJ=-\bmJ^{\top}\in\R^{\FOMsize\times \FOMsize}$  be skew-symmetric and let $\bmR=\bmR^{\top} \in\R^{\FOMsize\times\FOMsize}$ be positive semi-definite. We further assume that $\bmJ-\bmR$ is invertible and define
\begin{equation}
		\label{equ:GMG_defi_Jcheck_Rcheck}
				\Jcheck \coloneq \frac{1}{2}\left((\bmJ-\bmR)^{-1} - (\bmJ-\bmR)^{-\top}\right)\ \quad \mathrm{and}\quad\Rcheck \coloneq
				\frac{1}{2} \left(-(\bmJ-\bmR)^{-1} -(\bmJ-\bmR)^{-\top}\right).
			\end{equation}
			 With these definitions, we have 
			\begin{equation*}
				\Jcheck-\Rcheck=(\bmJ-\bmR)^{-1},\quad
				\Jcheck = -\Jcheck^{\top}, \quad \text{and}
				\quad\Rcheck = \Rcheck^{\top}.
			\end{equation*}
			We further have $\Rcheck \succeq 0$, since for any $\statex\in\R^{\FOMsize}$
			\begin{equation*}
				\begin{aligned}
					0&\leq
					\statex^{\top}
					(\bmJ-\bmR)^{-\top}
					\bmR
					(\bmJ-\bmR)^{-1}
					\statex
					=
					-\statex^{\top}
					(\bmJ-\bmR)^{-\top}
					(\bmJ-\bmR)
					(\bmJ-\bmR)^{-1}
					\statex\\
					&=
					-\statex^{\top}
					(\Jcheck-\Rcheck)^{\top}
					\statex
					=
					\statex^{\top}
					\Rcheck
					\statex.\\
	\end{aligned}	
\end{equation*}			

\begin{theorem}
\label{Theo:exists_JR_inv}
			Let $\bmJ=-\bmJ^{\top}\in\R^{\FOMsize\times \FOMsize}$ be a skew-symmetric matrix and let $\bmR=\bmR^{\top}\in\R^{\FOMsize\times\FOMsize}$ be a positive semi-definite matrix. If $\bmR\neq\bzero$, i.e., $\bmR$ is not the zero matrix, and the matrix $(\bmJ-\bmR)$ is invertible, then for almost every matrix $\bmU\in \R^{\FOMsize\times\ROMsize}, 1\leq\ROMsize\leq\FOMsize$, it holds that $\bmU\in\matset_{(\bmJ-\bmR)^{-1}}$, i.e., $\mathrm{det}\left(\bmU^{\top}(\bmJ-\bmR)^{-1}\bmU\right)\neq 0$.
\end{theorem}	
\begin{proof}
For any $1\le\ROMsize\leq \FOMsize$, let $f_{\mathrm{det},\ROMsize,1}\colon\R^{\FOMsize\times\ROMsize}\rightarrow \R$ be defined by 
			\begin{align*}
				f_{\mathrm{det},\ROMsize,1}\colon\bmU\mapsto \allowbreak \mathrm{det}\left(\bmU^{\top}\left(\Jcheck-\Rcheck\right)\bmU\right).
			\end{align*}
			Note that $f_{\mathrm{det},\ROMsize,1}$ is a polynomial with respect to the entries of $\bmU$. This, means that $f_{\mathrm{det},\ROMsize,1}$ is a real analytic function.
			By Lemma~\ref{lemma:analytic_measure}, it is sufficient to show that $f_{\mathrm{det},\ROMsize,1}$ is not identically zero. Indeed, if $f_{\mathrm{det},\ROMsize,1} \not\equiv 0$, then its zero set has Lebesgue measure zero. Therefore 
			\begin{equation*}
				\mathrm{det}\left(\bmU^{\top}\left(\Jcheck-\Rcheck\right)\bmU\right)\neq 0,
			\end{equation*}
			holds for almost every $ \bmU\in\R^{\FOMsize\times \ROMsize} $ for $\ROMsize=1,\ldots,\FOMsize$. It remains to construct such a $\bmU$ such that $\mathrm{det}\left(\bmU^{\top}\left(\Jcheck-\Rcheck\right)\bmU\right)\neq 0,$ The explicit construction is moved to  Appendix \ref{sec_app_proofs}. 
\end{proof}
					
\begin{theorem}
\label{Theo:exists_JR_inv_BVbar}
Let $\bmJ=-\bmJ^{\top}\in\R^{\FOMsize\times \FOMsize}$ be skew-symmetric, let $\bmR=\bmR^{\top} \in\R^{\FOMsize\times\FOMsize}$ be positive semi-definite with $\bmR \neq \bm{0}$ and assume that $(\bmJ-\bmR)$ is invertible. Let $\bmB\in\R^{\FOMsize\times\portsize}$ be of full column rank and define 
\begin{align*}
	\ell := \mathrm{rank}\left(\bmB, (\bmJ-\bmR)^{-1}\bmB \right).
\end{align*}
Furthermore, let $\Jcheck,\Rcheck \in \R^{\FOMsize \times \FOMsize}$ be defined as in \eqref{equ:GMG_defi_Jcheck_Rcheck}. Assume that 
\begin{align} \label{eq:full_space_assumption}
			 \mathrm{ker}(\Rcheck)\cup
			 \mathrm{span}(\bmB)\cup
			 \mathrm{span}((\bmJ-\bmR)^{-1}\bmB)
			 \neq
			 \R^{\FOMsize},
\end{align}			 
and that the matrices $\bmB^{\top}(\bmJ-\bmR)^{-1}\bmB$ and $\bmB^{\top}\bmR\bmB$ are invertible. Then, for almost every matrix $\bmU\in \R^{\FOMsize\times\ROMsize} (\ROMsize\leq\FOMsize-\ell)$, it holds that 
\begin{align*}
	\begin{bmatrix}
				\bmB,\bmU
			\end{bmatrix}\in\matset_{(\bmJ-\bmR)^{-1}},
\end{align*}
i.e., 
\begin{align*}
\mathrm{det}\left(\begin{bmatrix}
				\bmB,\bmU
			\end{bmatrix}^{\top}(\bmJ-\bmR)^{-1}\begin{bmatrix}
				\bmB,\bmU
			\end{bmatrix}\right)\neq 0.
\end{align*}
\end{theorem}
\begin{proof}
For any $1\le\ROMsize\leq \FOMsize-\ell$, let $f_{\mathrm{det},\ROMsize,2}\colon\R^{\FOMsize\times\ROMsize}\rightarrow \R$ be defined by 
			\begin{align*}
				f_{\mathrm{det},\ROMsize,2}(\bmU)\coloneq \mathrm{det}\left(\begin{bmatrix} \bmB,\bmU \end{bmatrix}^{\top} \left(\Jcheck-\Rcheck\right) \begin{bmatrix} \bmB,\bmU \end{bmatrix}\right).
			\end{align*}
We note that $f_{\mathrm{det},\ROMsize,2}$ is a polynomial with respect to the entries of $\bmU$, i.e. $f_{\mathrm{det},\ROMsize,2}$ is a real analytic function. By Lemma~\ref{lemma:analytic_measure}, it is sufficient to show that $f_{\mathrm{det},\ROMsize,2}$ is not identically zero. Equivalently, it remains to construct a matrix  $ \bmU\in\R^{\FOMsize\times \ROMsize} $ for $\ROMsize=1,\ldots,\FOMsize-\ell$ such that
			\begin{equation*}
				\mathrm{det}\left(\begin{bmatrix} \bmB,\bmU \end{bmatrix}^{\top} \left(\Jcheck-\Rcheck\right) \begin{bmatrix} \bmB,\bmU \end{bmatrix}\right) \neq 0. 
\end{equation*}
Let $\bmW\in\R^{\FOMsize\times(\FOMsize-\ell)}$ be a matrix whose columns form an orthogonal complement of \\$\textnormal{span}(\bmB,(\Jcheck-\Rcheck)\bmB)$. 
			This means we have 
			\begin{align}
				\label{equ:exists_JR_inv_BVbar_orth_comp}
				\bmB^\top\bmW = \bm{0},\
				\bmB^\top (\Jcheck-\Rcheck)^\top \bmW = \bm{0},\ \textnormal{ and }
				\mathrm{span}(
					\bmB, (\Jcheck-\Rcheck)\bmB, \bmW
				) =\R^{\FOMsize}.
			\end{align}
			Define
			$
			\bmJ_1 \coloneq  \bmW^{\top}\Jcheck\bmW$ and
			$\bmR_1 \coloneq  \bmW^{\top}\Rcheck\bmW.
			$
			Further, assume that $\mathrm{det}(\bmJ_1-\bmR_1)= 0$. Then, there exists a nonzero vector $\statex\in\R^k$ s.t.
				$(\bmJ_1-\bmR_1)\statex = \bzero$.
			From the latter equation, we obtain
			\begin{equation}
				\label{equ:GMG_exist_Orth_cond}
				\bmW^\top (\Jcheck-\Rcheck)\bmW\statex =\bm{0}\ \textnormal{ and }
				\statex^{\top}\bmW^{\top}\Rcheck\bmW\statex\allowbreak= 0.
			\end{equation}
			From \eqref{equ:exists_JR_inv_BVbar_orth_comp} and \eqref{equ:GMG_exist_Orth_cond}, we get 
			\begin{align} \label{eq:orthcond}(\Jcheck-\Rcheck)^{\top}\bmW\perp \textnormal{span}\left( \bmW\statex, \bmB, (\bmJ-\bmR) \bmB\right), \end{align}
			where we used used
			\begin{align*}
				\bm{0} = \bmB^\top\bmW = \bmB^\top (\bmJ-\bmR)^\top (\bmJ-\bmR)^{-\top}  \bmW = \bmB^\top (\bmJ-\bmR)^\top (\Jcheck-\Rcheck)^{\top}  \bmW.
			\end{align*}
			We have that the dimension of the space $\textnormal{span}(\bmB,(\Jcheck-\Rcheck)\bmB) = \textnormal{span}(\bmB,(\bmJ-\bmR)^{-1}\bmB)$. Then, due to $\textnormal{det}(\bmJ-\bmR)\neq 0$, it follows that the dimension of the space $\textnormal{span}((\bmJ-\bmR)\bmB, \bmB)$ is $\ell$. Further, since the dimension of the space $\textnormal{span} ((\Jcheck-\Rcheck)^{\top}\bmW)$ is $\FOMsize-\ell$, it follows due to the orthogonality conditions \eqref{eq:orthcond} that 				
			$$
				\mathrm{rank}\left(
				\begin{bmatrix}
					(\bmJ-\bmR)\bmB, \bmB, (\Jcheck-\Rcheck)^{\top}\bmW
				\end{bmatrix}
				\right) =\FOMsize.$$
				From the latter rank argument combined with \eqref{eq:orthcond}, we obtain $\bmW\statex\in\mathrm{span}\left((\bmJ-\bmR)\bmB\right)$. 
			
			If $\ell=\portsize$, then
			  we get $
			\mathrm{dim}\left(\textnormal{span}
				((\bmJ-\bmR)\bmB, \bmB))\right)
			=\portsize.
			$ 
			Since the matrix $\bmB$ is assumed to have full column rank, we arrive at $\bmW\statex\in \mathrm{span}
			\left((\bmJ-\bmR)\bmB\right)\subseteq
			\mathrm{span}(\bmB).
			$
			This yields $\bmW\statex\in \textnormal{span}(\bmW)\cap \textnormal{span}( \bmB)= \{ \bm{0} \}.$
			Thus, we have $\bmW\statex=\bm{0}$, yielding $\statex=\bm{0}$. However, this is a contradiction since we assume $\statex\neq\bm{0}$ and it follows that $\mathrm{det}(\bmJ_1-\bmR_1)\neq0. $
			
		    We now consider the remaining case $\ell\neq\portsize$.
		    We know that  $\exists\stateapprox\in\R^{\portsize}$ s.t. $\bmW\statex=(\bmJ-\bmR)\bmB\stateapprox$ and with \eqref{equ:GMG_exist_Orth_cond} we arrive at
			\begin{align*}
				0 =&\statex^{\top}\bmW^{\top}\Rcheck\bmW\statex
				=\stateapprox^{\top}\bmB^{\top}(\bmJ-\bmR)^{\top}
				\Rcheck
				(\bmJ-\bmR)\bmB\stateapprox\\
				=&
				-\stateapprox^{\top}\bmB^{\top}(\bmJ-\bmR)^{\top}
				(\Jcheck-\Rcheck)
				(\bmJ-\bmR)\bmB\stateapprox\\
				=&
				-\stateapprox^{\top}\bmB^{\top}(\bmJ-\bmR)^{\top}
				\bmB\stateapprox
				=
				\stateapprox^{\top}\bmB^{\top}\bmR
				\bmB\stateapprox.
			\end{align*}
			Since $\bmB^{\top}\bmR\bmB$ is non-degenerate and $\bmW$ is of full column rank, we get $\statex= \bzero$, which is again a contradiction to our assumption of a nonzero $\statex$. Also in this case we arrive at $\mathrm{det}(\bmJ_1-\bmR_1)\neq0.$
			
Recall the assumption \eqref{eq:full_space_assumption}. Since $\textnormal{span}(\bmW)$ is by construction the orthogonal complement of $\mathrm{span}(
					\bmB, (\Jcheck-\Rcheck)\bmB)$ \eqref{equ:exists_JR_inv_BVbar_orth_comp}, it follows that $\textnormal{span}(\bmW) \not\subseteq \mathrm{ker}(\Rcheck)$. Thus, there exists $\bm{w} \in \bmW$ such that $\bm{w}^{\top} \Rcheck \bm{w} >0$ and we obtain $\bmR_1 = \bmW^{\top}\Rcheck\bmW \neq \bzero$.	
					
Since $(\bmJ_1-\bmR_1)$ is non-degenerate and $\bmR_1 \neq \bzero$, we can construct $\Jcheck_1,\Rcheck_1\in\R^{(\FOMsize-\ell)\times(\FOMsize-\ell)}$ based on \eqref{equ:GMG_defi_Jcheck_Rcheck} s.t.
			\begin{equation*}
				\Jcheck_1-\Rcheck_1 = (\bmJ_1-\bmR_1)^{-1},\quad
				\Jcheck_1=-\Jcheck_1^{\top},\quad
				\Rcheck_1 =\Rcheck_1^{\top}\succeq 0,\quad
				\Rcheck_1\neq \bzero.				
			\end{equation*}
By applying Theorem \ref{Theo:exists_JR_inv} to $\Jcheck_1, \Rcheck_1$, we know that for almost every matrix $\bmK\in\R^{(\FOMsize-\ell)\times \ROMsize}$, $1 \leq \ROMsize\leq(\FOMsize-\ell)$, it holds that $\mathrm{det}(\bmK^{\top}(\Jcheck_1-\Rcheck_1)^{-1}\bmK)\neq 0$.
			Finally, we select $\bmK\in\matset_{(\Jcheck_1-\Rcheck_1)^{-1}}$, set $\bmU= \bmW\bmK$ and we can show that $[\bmB,\bmU]\in\matset_{(\bmJ-\bmR)^{-1}}$ by the following equation
			\begin{align*}
				&\mathrm{det}\left(
				\begin{bmatrix}
					\bmB,\bmW\bmK
				\end{bmatrix}^{\top}
				(\bmJ-\bmR)^{-1}
				\begin{bmatrix}
					\bmB,\bmW\bmK
				\end{bmatrix}
				\right)
				\\
				=
				&
				\mathrm{det}\left(
				\begin{bmatrix}
					\bmB^{\top}	(\bmJ-\bmR)^{-1}\bmB&
					\bmB^{\top}	(\bmJ-\bmR)^{-1}\bmW\bmK\\
					\bmK^{\top}\bmW^{\top}(\bmJ-\bmR)^{-1}\bmB&
					\bmK^{\top}(\bmW^{\top}(\bmJ-\bmR)^{-1}\bmW)\bmK
				\end{bmatrix}
				\right)
				\\
				=
				&
				\mathrm{det}\left(
				\begin{bmatrix}
					\bmB^{\top}	(\bmJ-\bmR)^{-1}\bmB&
					\bmB^{\top}	(\bmJ-\bmR)^{-1}\bmW\bmK\\
					\bzero&
					\bmK^{\top}(\bmJ_1-\bmR_1)\bmK
				\end{bmatrix}
				\right)
				\\
				=
				&
				\mathrm{det}\left(\bmB^{\top}	(\bmJ-\bmR)^{-1}\bmB\right)
				\mathrm{det}\left(\bmK^{\top}(\Jcheck_1-\Rcheck_1)^{-1}\bmK\right)
				\neq
				0.
\end{align*}							
\end{proof}					
		
\begin{remark}
Theorem \ref{Theo:exists_JR_inv_BVbar} extends Proposition 2.1 of \cite{GruberTezaur2025} from Hamiltonian systems to \pH systems. In particular, the condition $\mathrm{det}\left(\bmU^{\top}\bmJ\bmU\right)\neq 0$ is replaced by the condition $\mathrm{det}\left(\bmU^{\top}(\bmJ-\bmR)^{-1}\bmU\right)\neq 0$ allowing for arbitrary skew-symmetric $\bmJ$, positive semi-definite $\bmR$ and even and odd dimensions of $\ROMsize$.
\end{remark}					
								
\subsection{State-approximations and resulting \ROMs} \label{Sec:state-approx_resulting_ROMs}
To determine the unknown matrices and functions of the approximation map $\approxfunc$, we use a data-driven approach. To this end, we introduce an equidistant discretization of the time interval $[t_0, t_{\text{end}}]$ with step size $\Delta t := (t_{\text{end}} - t_0)/\snapsize$, where $\snapsize \in \mathbb{N}$. Then, for $t_i \coloneq  t_0 + i \Delta t $, $0 \le i \le \snapsize$, we seek approximations $\statex^i \approx \statex(t_i)$, $\outputy^i \approx \outputy(t_i)$.

Evaluating the \pH system \eqref{equ:Nonlinear_FOM_pH} at these time instances yields the snapshot matrix   
\begin{align}
\label{equ:snap_mat_difi}
	\bmX : =
	\begin{bmatrix}
		\statex^0,&\ldots,&\statex^{\snapsize} 
	\end{bmatrix}\in\R^{\FOMsize\times (\snapsize+1)}.
\end{align}
In the following, we discuss particular realizations of the approximation map \eqref{equ:nonlinear_approximation}  leading to (i) linear (Section \ref{Sec:GMGPODROM}) and (ii) quadratic approximation maps (Section \ref{Sec:GMGQMROM}), together with the corresponding \ROMs.
								
								
\subsubsection{Linear state approximation and \GMG-\POD-\ROM} 
\label{Sec:GMGPODROM}
We choose $\eta$ to be the identity map on $\R^{r-m}$, i.e., $\eta: \R^{r-m} \rightarrow \R^{r-m}$ with $\eta(\stateRed_2)= \stateRed_2$. 
Then, the representation in \eqref{equ:general_embedding_map} reduces to the linear approximation map
								\begin{align}
									\label{equ:linear_approximation}
									\approxfunc_{\text{lin}}(\stateRed)=\bmV\stateRed =\bmB\stateRed_1+\VPODGMG\stateRed_2,
								\end{align}
								with $\bmV= 
								\begin{bmatrix}
									\bmB,&\VPODGMG
								\end{bmatrix}
\in \R^{\FOMsize \times \ROMsize}$. We assume that $\bmV$ has full column rank. This is a mild assumption since full column rank of $\VPODGMG$ can be ensured by construction. Moreover, if $\bmB$ is not of full column rank, the \pH system can be rewritten in a form with less ports. Further, as $\bmB$ is given by the problem formulation, we only need to find a suitable $\VPODGMG$ such that the conditions of Theorem \ref{Theo:pH_structure} are fulfilled. This is achieved by solving the following optimization problem where ${(\cdot)}^{\dagger}$ represents the Moore-Penrose pseudoinverse 
\begin{align*}
		\VPODGMG
		&=
		\argmin_{\bmA\in\Xi_{\VPODGMG}}
		\sum_{i=0}^{\snapsize}
		\left\| \statex^i-\begin{bmatrix}
			\bmB,\bmA
		\end{bmatrix}
		\begin{bmatrix}
			\bmB,\bmA
		\end{bmatrix}^{\dagger}\statex^i \right\|_2^2\\
		&=
		\argmin_{\bmA\in\Xi_{\VPODGMG}}
		\left\| 
		\left(\mathbb{I}_N-\bmA\bmA^{\top}\right)
		\left(\bmX-\bmB\bmB^{\dagger}\bmX\right)
		\right\|_F^2 ,
\end{align*}
where 
	$\Xi_{\VPODGMG} = \left\{\bmA\left|\right.\bmA\in\R^{\FOMsize\times(\ROMsize-\portsize)},
	\
	 \bmB^{\top} \bmA= \bm{0}_{\portsize \times (\ROMsize-\portsize)},
	 \
	 \bmA^{\top}\bmA=\mathbb{I}_{\ROMsize-\portsize}\right\}.$
This means that $\VPODGMG$ can be determined by applying \POD to the adjusted snapshot matrix $\bmX-\bmB\bmB^{\dagger}\bmX$.
For the point reduction map used to project the initial condition, we choose 
\begin{align}
	\rho_{\text{lin}}(\statex) := \bmV^{\dagger} \statex= 
	\begin{bmatrix}
		\bmB^{\dagger},&\Vbar^{\top}
	\end{bmatrix}^{\top} \statex,
\end{align}
which, by construction, satisfies the point projection property, i.e., $\rho_{\text{lin}} \circ \approxfunc_{\text{lin}} = \text{id}_{\R^{r}}$. The \ROM obtained from $\approxfunc_{\text{lin}}$, $\rho_{\text{lin}}$ and the corresponding \GMG reduction is called \GMG-\POD-\ROM. 

\begin{remark}
 The energy-stable \MOR method for \pH systems\footnote{Not to be confused with energy-stable \pH systems, which have been recently introduced in \cite{BUCHFINK2026109784}.} introduced in \cite{Rettberg31122023} appears similar to the \GMG-\POD-\ROM since both methods involve $(\bmJ-\bmR)^{-1}$ in the projection matrix. However, the method in \cite{Rettberg31122023} preserves only the state equation, whereas the corresponding output equation is not preserved. Thus, the resulting \ROM does not in general inherit the full \pH structure.  
\end{remark}


\subsubsection{Quadratic state approximation and \GMG-\QM-\ROM}
\label{Sec:GMGQMROM}
We now consider approximation maps on quadratically embedded manifolds of the form
	\begin{align}
		\label{equ:quadratic_approximation_with_M}
		\approxfunc_{\text{quad}}(\stateRed)
		=\bmB\stateRed_1+\VPODQMone\stateRed_2 + \VPODQMtwo\bmM(\stateRed_2\otimes \stateRed_2),
		\quad
		\stateRed_1\in\R^{\portsize},
		\quad
		\stateRed_2\in\R^{\ROMsize-\portsize},
	\end{align}
where $\nlsize\leq (\ROMsize-\portsize)^2$ denotes the lifting dimension, $\otimes$ the Kronecker product, $	\VPODQMone\in\R^{\FOMsize\times(\ROMsize-\portsize)}$,
$\VPODQMtwo\in\R^{\FOMsize\times \nlsize}$, and $\bmM\in\R^{\nlsize\times(\ROMsize-\portsize)^2}$.
This corresponds to $\eta : \R^{\ROMsize-\portsize} \rightarrow \R^{\nlsize+\ROMsize-\portsize}$ with $\eta( \stateRed_2)= [ \stateRed_2, \bmM(\stateRed_2\otimes \stateRed_2)]$ and $\bmV= [\VPODQMone, \VPODQMtwo]$ in \eqref{equ:general_embedding_map}.
We next describe the construction of $\VPODQMone, \VPODQMtwo$, and $\bmM$. 
First, we construct $\VPODQMone$ and $\VPODQMtwo$ by a hierarchical POD procedure based on the adjusted snapshot matrix $\bmX-\bmB\bmB^{\dagger}\bmX$, where $\bmX\in\R^{\FOMsize\times(\snapsize +1)}$ denotes the snapshot matrix introduced in \eqref{equ:snap_mat_difi}. The basis $\VPODQMone$ is obtained by applying \POD of $\bmX-\bmB\bmB^{\dagger}\bmX$, while $\VPODQMtwo$ is computed from the remaining residual, i.e., 
	\begin{align}
		\label{equ:construc_V_comb_A}
		\VPODQMone &= \argmin_{\bmA\in\R^{\FOMsize\times (\ROMsize-\portsize)},
			\bmA^{\top}\bmA=\mathbb{I}_{\ROMsize-\portsize}}
		\left\|\left(\mathbb{I}_{\FOMsize}-\bmA\bmA^{\top} \right) \left( \bmX-\bmB\bmB^{\dagger}\bmX \right) \right\|_F^2, \qquad \text{ and }& 
		\\
		\label{equ:construc_V_comb_B}
		\VPODQMtwo
		&= \argmin_{\bmA\in\R^{\FOMsize\times \nlsize},
			\bmA^{\top}\bmA=\mathbb{I}_{\nlsize}}
		\left\| \left(\mathbb{I}_{\FOMsize}-\bmA\bmA^{\top}\right)\left(\bmX-
		\begin{bmatrix}
			\bmB, \VPODQMone
		\end{bmatrix}
		\begin{bmatrix}
			\bmB, \VPODQMone
		\end{bmatrix}^{\dagger}
		\bmX\right)\right\|_F^2.&
	\end{align}
	We now describe how $\bmM$ can be computed. To this end, we introduce the point reduction $\rho_{\text{quad}}(\statex)=
	\begin{bmatrix}
		\bmB,\VPODQMone
	\end{bmatrix}^{\dagger} \statex.
	$
	Since
	$\VPODQMtwo\perp
	\begin{bmatrix}
		\bmB,\VPODQMone
	\end{bmatrix}$, the point projection property is satisfied, i.e., $\rho_{\text{quad}} \circ \approxfunc_{\text{quad}}=\text{id}_{\R^{r}}.$ We therefore compute the reduced coordinates of the snapshots by $\begin{bmatrix}\stateRed_1^i, \stateRed_2^i\end{bmatrix}^\top= 		\begin{bmatrix}
		\bmB,\VPODQMone
	\end{bmatrix}^{\dagger} \statex^i$, $i=1, \ldots, \snapsize$. 
	Next, we define the residual associated with the linear approximation by 
	\begin{align*}
		\textnormal{res}^{i}_{\rm{QM}} := \statex^i - \bmB\bmB^{\dagger}\statex^i-\VPODQMone\VPODQMone^{\top}\statex^i, \quad i=1,\dots,\snapsize.
	\end{align*}
	The matrix $\bmM$ is then determined by fitting the quadratic correction $\VPODQMtwo\bmM(\stateRed_2^i\otimes\stateRed_2^i)$ to the residual $\textnormal{res}^{i}_{\mathrm{QM}}$. This yields the regularized least squares problem
	\begin{align}
		\label{equ:least_square_reg}
		\bmM = \argmin_{\bmA\in\R^{\nlsize\times (\ROMsize-\portsize)^2}}\left(
		\sum_{i=1}^{\snapsize}
		\left\|
		\textnormal{res}_{\rm{QM}}^{i}
		-
		\VPODQMtwo\bmA (\statex_2^i \otimes\statex_2^i)
		\right\|_2^2
		+\lambda_{\rm reg} \|\bmA\|_F^2
		\right),
	\end{align}
	where $\lambda_{\rm reg} \in\R^{+}$ denotes a regularization parameter. 
	 
	The \ROM obtained by $\approxfunc_{\text{quad}}$, $\rho_{\text{quad}}$ and the corresponding \GMG reduction map is again a \pH system. 
	We refer to the resulting model as \GMG-\QM-\ROM. 

\section{Numerical results}
\label{Sec:Numerical_results}

In this section, we evaluate the \MOR methods introduced in Sections \ref{Sec:GMGPODROM} and \ref{Sec:GMGQMROM} and compare the performance with established structure-preserving \MOR methods for \pH systems. 
Section \ref{subsec:Num_imp} provides implementation details and introduces the error measures used throughout the numerical experiments. 
In Section \ref{subsec:Num_linear_MSD}, we consider a linear mass-spring-damper system, followed by a nonlinear mass-spring-damper system in Section \ref{subsec:Num_nonlinear_MSD}.

\subsection{Implementation details and error measures}
\label{subsec:Num_imp}

In this paper, all numerical simulations are conducted using Python 3.11. 
For temporal integration, we use the sixth-order Gauss-Legendre method, as it preserves the \pH structure \cite{kotyczka2018discrete}.  
For nonlinear \pH systems, the resulting implicit time-stepping scheme is solved using Newton's method at each time step.
We compare the proposed \MOR methods to existing structure-preserving \MOR methods for \pH systems from the literature. For the linear mass-spring-damper system, we consider the method introduced in \cite{schulze2023structure} and refer to the resulting \ROM as \SPLin-\POD-\ROM. 
For the nonlinear mass-spring-damper system, which involves a nonlinear Hamiltonian, we compare against the structure-preserving \MOR method presented in \cite{chaturantabut2016structure}, and refer to the resulting \ROM \SPQuad-\POD-\ROM. 
We briefly recall both methods below. 

The \SPLin-\POD-\ROM uses the linear approximation map $\approxfunc(\stateRed)=\VPOD\stateRed$, where $\VPOD\in\R^{\FOMsize\times\ROMsize}$ is constructed from the first $r$ \POD modes of the snapshot matrix $\bmX$ introduced in \eqref{equ:snap_mat_difi}, i.e.,
\begin{align}
	\label{equ:const_Lin_POD}
	\VPOD= \argmin_{\bmA^{\top}\bmA=\mathbb{I}_{\ROMsize}}
	\left\|
	\bmX - \bmA\bmA^{\top}\bmX
	\right\|^2_{F}.
\end{align}
The reduction map $\mathcal{R}_{\text{red}}(\statex) = \bmW^\top\statex$ is constructed with $\bmW\in\R^{\FOMsize\times\ROMsize}$, obtained from a \GMG reduction with matrix $\bmQ$, 
\begin{align}
	\label{equ:const_Lin_POD_red}
	\bmW = \matfunc_{Q}\left(\VPOD\right)^{\top}
	=
	\bmQ\bmV\left(\bmV^{\top}\bmQ\bmV\right)^{-1},
\end{align}
 where $\bmQ=\nabla_{\statex}^{2}\curlyH$ denotes the Hessian matrix of a quadratic Hamiltonian.
 
 The same linear approximation map is used in the construction of the \SPQuad-\POD-\ROM. However, a different reduction map is applied.
 First, a matrix $\bmV_{\nabla_{\statex}\curlyH}\in\R^{\FOMsize\times\ROMsize}$ is constructed from the first $r$ \POD modes of  
\begin{align*}
 	\bmX_{\nabla_{\statex}\curlyH}\coloneqq
 	\begin{bmatrix}
 		\nabla_{\statex}\curlyH(\statex^0),
 		\allowbreak
 		\ldots,
 		\allowbreak
 		\nabla_{\statex}\curlyH(\statex^{\snapsize})
 	\end{bmatrix},
 \end{align*}
i.e., the snapshot matrix of Hamiltonian gradients. 
 The linear reduction map $\mathcal{R}_{\text{red}}(\statex)=\bmW^{\top}\statex$ is then constructed with $\bmW\in\R^{\FOMsize\times\ROMsize}$ given by
 \begin{align}
 	\label{equ:const_SP_POD_redmap}
 	\bmW=\left(\bmV_{\nabla_{\statex}\curlyH}^{\top}\VPOD\right)^{-1}\bmV_{\nabla_{\statex}\curlyH}.
 \end{align} 

To summarize the main differences between the considered methods, Table \ref{table:proj_bases} lists the corresponding approximation and reduction maps together with the resulting \ROMs. 

\begin{table}[htbp]
\setlength{\tabcolsep}{4pt}
	\begin{center}
		\begin{tabular}{|c|c|c|}
			\hline
			\ROM  & Approximation map  $\approxfunc$ & 
			Reduction map $\mathcal{R}_{\text{red},\statex}$
			\\ 
			\hline
			\SPLin-\POD-\ROM \cite{schulze2023structure} 
			& $\VPOD\stateRed$ 
			\eqref{equ:const_Lin_POD}
			& $\matfunc_{Q}(\VPOD)^{\top}\statex$
			\eqref{equ:const_Lin_POD_red}
			\tstrut 
			\\
			\SPQuad-\POD-\ROM \cite{chaturantabut2016structure} 
			& $ \VPOD\stateRed$ 
			\eqref{equ:const_Lin_POD}
			& $(\bmV_{\nabla_{\statex}\curlyH}^{\top}\VPOD)^{-1}\bmV_{\nabla_{\statex}\curlyH}\statex$
			\eqref{equ:const_SP_POD_redmap}
			\\
			\hline
			\GMG-\POD-\ROM& 
			$
			\begin{bmatrix}
				\bmB,\VPODGMG
			\end{bmatrix}\stateRed$
			\eqref{equ:linear_approximation}
			& $\matfunc_{(\bmJ-\bmR)^{-1}}(\bmD_{\stateRed}\approxfunc(\stateRed))^{\top}\statex$
			\eqref{equ:GMG_NL_Construct_W}
			\tstrut \\
			\GMG-\QM-\ROM & 
			$ 
			\begin{bmatrix}
				\bmB,\VPODQMone
			\end{bmatrix}\stateRed + \VPODQMtwo\bmM(\stateRed_2\otimes\stateRed_2)$ 
			\eqref{equ:quadratic_approximation_with_M}& $\matfunc_{(\bmJ-\bmR)^{-1}}(\bmD_{\stateRed}\approxfunc(\stateRed))^{\top}\statex$
			\eqref{equ:GMG_NL_Construct_W}
			\\
			\hline
		\end{tabular}
		\caption{Approximation and reduction maps of different MOR methods and resulting ROMs.}
		\label{table:proj_bases}
	\end{center}
\end{table}

To compare the performance of the considered \MOR methods, we introduce error measures for the state, the output and the preservation of the energy balance equation. 
Let $\statex^0,\ldots,\statex^{\snapsize}$ and $
	\stateRed(t_0),\ldots,\stateRed(t_{\snapsize})$ denote the trajectories of the \FOM and the corresponding \ROM.
We denote the relative state reduction error $e_{\statex, {\rm red}}$ and the relative state projection error $e_{ \statex, {\rm proj}}$ by 
\begin{align}\label{eq:red_error}
	e_{\statex, {\rm red}} &\coloneq 
	\sqrt{
		\frac{\sum_{i=0}^{\snapsize}
			\|\statex^i-\approxfunc(\stateRed(t_i))\|_2^2
		}
		{
			\sum_{i=0}^{\snapsize}
			\|\statex^i\|_2^2
		}
	},
	\\ \label{eq:proj_error}
	e_{\statex, {\rm proj}}  &\coloneq 
	\sqrt{
		\frac{\sum_{i=0}^{\snapsize}
			\|\statex^i-\approxfunc(\rho(\statex^i))\|_2^2
		}
		{
			\sum_{i=0}^{\snapsize}
			\|\statex^i\|_2^2
		}
	}.
\end{align}
For nonlinear approximation maps $\approxfunc$ of the form \eqref{equ:quadratic_approximation_with_M}, we additionally consider the lower bound for the reduction error proposed in \cite{buchfink2024approximation}, given by 
\begin{align}
	\label{equ:non_app_lower_bound}
	e_{\statex,{\rm lowerbound}} \coloneq 
	\sqrt{
		\left(\sum_{i=0}^{\snapsize}
		\left\|
		(\mathbb{I} - \bmB\bmB^{\dagger}-\VPODQMone\VPODQMone^{\top} - \VPODQMtwo \VPODQMtwo^{\top})
		\statex^i
		\right\|_2^2
		\right)
		\Big/
		\left(
		\sum_{i=0}^{\snapsize}
		\|\statex^i\|_2^2
		\right)
	}
	.
\end{align}
The average relative output error is defined by
\begin{align}\label{eq:output_error}
	e_{\outputy} \coloneq 
	\sqrt{
		\frac{\sum_{i=0}^{\snapsize}
			\|\outputy^i-\outputRed(t_i)\|_2^2
		}
		{
			\sum_{i=0}^{\snapsize}
			\|\outputy^i\|_2^2
		}
	},
\end{align}
where $\outputy$ and $\outputRed$ denote the outputs of the \FOM and the \ROM, respectively. 
Furthermore, we evaluate how well the energy balance equation \eqref{equ:energy_conserve_time_int} is preserved by the considered \MOR methods. To this end, we define this energy balance error by  
\begin{align}
	\label{equ:def_energy_conserve_error}
	e_{\rm energy}({t_{\rm{end}}})&\coloneq \left| \curlyH\left(\statex({t_{\rm{end}}})\right)-\curlyH\left(\statex(0)\right)- {\int_{t=0}^{t_{\rm{end}}} {\outputyT^{\top}\inputu(t) \ {\rm d}t}} \right. \\
		 &+\left.	\int_{t=0}^{t_{\rm{end}}}\left(\funcR\nabla_{\statex}\curlyH(\statexT)\right)^{\top}\nabla_{\statex}\curlyH(\statexT) \ {\rm d}t . \notag \right| 
\end{align}

\subsection{Linear mass-spring-damper system}
\label{subsec:Num_linear_MSD}

We consider the linear mass-spring-damper system from \cite{gugercin2012structure}, depicted in Figure \ref{fig:msd_sys}. 
For $i=1,\ldots,\frac{\FOMsize}{2}-1$, the masses $m_i$ and $m_{i+1}$ are connected by a spring with stiffness coefficient $k_i$.  
The last mass $m_{\frac{\FOMsize}{2}}$ is connected to a wall by a spring with stiffness coefficient $k_{\frac{\FOMsize}{2}}$. 
 Dampers are attached to all masses, and the damping coefficient associated with mass $m_i$ is $c_i$ for $i=1,\ldots,\frac{\FOMsize}{2}$.
\begin{figure}[H]
	\centering
	\begin{circuitikz}[scale = 0.7][H]
		\pattern[pattern=north east lines] (13,0.5) rectangle (13.25,3);
		\draw[thick] (13,0.5) -- (13,3);
		\draw (11,2.25) to[spring, l=$k_{\frac{\FOMsize}{2}}$] (13,2.25);
		\draw (11,1.25) to[damper, l_=$c_{\frac{\FOMsize}{2}}$] (12,1.25);
		\pattern[pattern=north east lines] (12,1) rectangle (12.125,1.5);
		\draw[fill=gray!40] (9,1) rectangle (11,2.5);
		\node at (10,1.75) {$m_{\frac{\FOMsize}{2}}$};
		
		\node at (8.4, 1.75) {\Large$\cdots$};
		
		\draw (6,2.25) to[spring, l=$k_{2}$] (8,2.25);
		\draw (6,1.25) to[damper, l_=$c_{2}$] (7,1.25);
		\pattern[pattern=north east lines] (7,1) rectangle (7.125,1.5);
		\draw[fill=gray!40] (4,1) rectangle (6,2.5);
		\node at (5,1.75) {$m_{2}$};
		
		\draw (2,2.25) to[spring, l=$k_{1}$] (4,2.25);
		\draw (2,1.25) to[damper, l_=$c_{1}$] (3,1.25);
		\pattern[pattern=north east lines] (3,1) rectangle (3.125,1.5);
		\draw[fill=gray!40] (0,1) rectangle (2,2.5);
		\node at (1,1.75) {$m_{1}$};
		
		\draw[->] (-1,1.75) -- (0,1.75);
		\node at (-0.7,2.25) {$u(t)$};
	\end{circuitikz}
	\caption{Linear mass-spring-damper system}
	\label{fig:msd_sys}
\end{figure}
A realization of this system for order $\FOMsize=6$ is 
\begin{align*}
	\bmB = 
	\begin{bmatrix}
		0\\
		1\\
		0\\
		0\\
		0\\
		0
	\end{bmatrix},
	\quad
	\bmJ =
	\begin{bmatrix}
		\jtwo & \bzero &\bzero \\
		\bzero &\jtwo  &\bzero \\
		\bzero & \bzero &\jtwo  
	\end{bmatrix},
	\quad
	\bmR =
	{\rm diag}(0,c_1,0,c_2,0,c_3),
\end{align*}
and Hamiltonian
\begin{align*}
		\curlyH\colon(\statex)=
	\frac{1}{2}\statex^{\top}\bmQ\statex,
\end{align*}
with
\begin{align*}
	\bmQ =
	\begin{bmatrix}
		k_1&0 & -k_1&0 &0 &0\\
		0&\frac{1}{m_1} & 0&0 &0 &0\\
		-k_1&0 & k_1+k_2&0 &-k_2 &0\\
		0&0 & 0&\frac{1}{m_2} &0 &0\\
		0&0 & -k_2&0 &k_2+k_3 &0\\
		0&0 & 0&0 &0 &\frac{1}{m_3}
	\end{bmatrix}.
\end{align*}
Here, 
\begin{align*}\jtwo=\begin{bmatrix}
	0&1\\
	-1&0
\end{bmatrix},
\end{align*} 
denotes the canonical Poisson matrix in two dimensions. 
We consider a 100-dimensional mass-spring-damper system where $m_i=2, k_i=1,$ and $c_i=1$ for $ i=1,2,\ldots,50$.  The time interval $(0,100\mathrm{s}]$ is discretized using a uniform time step $\Delta t = 0.1$s, resulting in $n_t = 1001$ time instances. We consider two different inputs $\inputu(t)= 0.1$ and $\inputu(t)=0.1\sin(t)$. 
For the \GMG-\QM-\ROM, we choose $\nlsize = 8$, and $\lambda_{\rm reg}=10^{-3}$. 
The tolerance and maximum number of iterations of the Newton's method are set to $10^{-10}$ and $10$, respectively.

In Figures \ref{fig:MSD_time_state_sin} and \ref{fig:MSD_time_state_cons}, the relative state reduction and projection errors are shown. 
The results demonstrate that the \GMG-based \MOR methods introduced in this paper achieve higher accuracy than the \SPLin-\POD-\ROM.
Moreover, the \GMG-\QM-\ROM consistently outperforms the \GMG-\POD-\ROM, highlighting the benefit of the quadratic approximation map.
For the \GMG-\POD-\ROM, the reduction error is only slightly larger than the corresponding projection error.
In addition, the reduction error of the \GMG-\QM-\ROM is close to the lower bound \eqref{equ:non_app_lower_bound}, indicating that the quadratic approximation captures the dominant features of the solution manifold.
The relative output errors are shown in Figure \ref{fig:MSD_time_output_sin} and \ref{fig:MSD_time_output_cons}. The results show a similar trend to the state error. The \GMG-based methods provide higher accuracy than the \SPLin-\POD-\ROM, while the \GMG-\QM-\ROM again yields the most accurate results among the considered methods.

\begin{figure}[htbp]
	\centering 
	\subfloat[$u(t) = 0.1\sin (t)$]{
		\begin{tikzpicture}
			\label{fig:MSD_time_state_sin}
			\begin{semilogyaxis}
				[
				width=.35\linewidth,
				height=.20\linewidth, 
				scale only axis,
				grid=both,
				grid style={line width=.1pt, draw=gray!10},
				major grid style={line width=.2pt, draw=gray!50},
				axis lines*=left,
				axis line style={line width=\AxisLineWidth},
				xmin=3.5, xmax=20.5,
				ymin=2e-7, ymax=5e-1,
				ytick={1e-7, 1e-6, 1e-5, 1e-4, 1e-3, 1e-2, 1e-1, 1e0},
				yticklabels = { , ,  $10^{-5}$ , , $10^{-3}$, , $10^{-1}$, },
				ylabel= $e_{\statex}$,
				xlabel= Red. order $r$,
				xlabel style={at={(axis description cs:0.5,-0.1)},anchor=north},
				ylabel style={at={(axis description cs:-0.2,0.5)},anchor=south},
				legend columns=3,
				legend entries={\SPLin-\POD-\ROM, \GMG-\POD-\ROM, \GMG-\QM-\ROM,
					$e_{\statex, {\rm proj}}$, $e_{\statex,{\rm lowerbound}}$},
				legend to name= Legend_Fig_MSD_time,
				unbounded coords=jump, 
				]
				\addplot 
				[color=\ColorSPG, mark=\MarkerSPG, \LineStyleSPG, mark size=\MarkerSize, line width=\LineWidth,mark options={solid}]
				table[x index=0, y index=1, col sep=comma]  {./Code/MSD/Data/sin/Plot/POD_SPG_state_error.dat};
				\addplot 
				[color=\ColorGMG, mark=\MarkerGMG,\LineStyleGMG, mark size=\MarkerSize, line width=\LineWidth,mark options={solid}]
				table[x index=0, y index=1, col sep=comma]  {./Code/MSD/Data/sin/Plot/POD_GMG_state_error.dat};
				\addplot 
				[color=\ColorGMGQM, mark=\MarkerGMGQM,\LineStyleGMG, mark size=\MarkerSize, line width=\LineWidth,mark options={solid}]
				table[x index=0, y index=3, col sep=comma]  {./Code/MSD/Data/sin/Plot/NL_state_error.dat};
				\addplot 
				[color = \ColorGMG, mark=\MarkerProj, \LineStyleProj, mark size=\MarkerSize, line width=\LineWidth,mark options={solid}]
				table[x index=0, y index=1, col sep=comma]  {./Code/MSD/Data/sin/Plot/MSD_linear_proj_error_sin.dat};
				\addplot 
				[color=\ColorLowbound, mark=\MarkerProj, \LineStyleProj, mark size=\MarkerSize, line width=\LineWidth,mark options={solid}]
				table[x index=0, y index=3, col sep=comma]  {./Code/MSD/Data/sin/Plot/MSD_nonlinear_proj_error_lower_bound_sin.dat};
			\end{semilogyaxis}
		\end{tikzpicture}
	}
	\subfloat[$u(t) = 0.1$]{
		\begin{tikzpicture}
			\label{fig:MSD_time_state_cons}
			\begin{semilogyaxis}
				[
				width=.35\linewidth,
				height=.20\linewidth, 
				scale only axis,
				grid=both,
				grid style={line width=.1pt, draw=gray!10},
				major grid style={line width=.2pt, draw=gray!50},
				axis lines*=left,
				axis line style={line width=\AxisLineWidth},
				xmin=3.5, xmax=20.5,
				ymin=2e-7, ymax=5e-1,
				ytick={1e-7, 1e-6, 1e-5, 1e-4, 1e-3, 1e-2, 1e-1, 1e0},
				yticklabels = { , ,  $10^{-5}$ , , $10^{-3}$, , $10^{-1}$, },
				yticklabels=\empty, 
				xlabel= Red. order $r$,
				xlabel style={at={(axis description cs:0.5,-0.1)},anchor=north},
				]
				\addplot 
				[color=\ColorSPG, mark=\MarkerSPG, \LineStyleSPG, mark size=\MarkerSize, line width=\LineWidth,mark options={solid}]
				table[x index=0, y index=1, col sep=comma]  {./Code/MSD/Data/constant/Plot/POD_SPG_state_error.dat};
				\addplot 
				[color=\ColorGMG, mark=\MarkerGMG,\LineStyleGMG, mark size=\MarkerSize, line width=\LineWidth,mark options={solid}]
				table[x index=0, y index=1, col sep=comma]  {./Code/MSD/Data/constant/Plot/POD_GMG_state_error.dat};
				\addplot 
				[color=\ColorGMGQM, mark=\MarkerGMGQM,\LineStyleGMG, mark size=\MarkerSize, line width=\LineWidth,mark options={solid}]
				table[x index=0, y index=3, col sep=comma]  {./Code/MSD/Data/constant/Plot/NL_state_error.dat};
				\addplot 
				[color = \ColorGMG, mark=\MarkerProj, \LineStyleProj, mark size=\MarkerSize, line width=\LineWidth,mark options={solid}]
				table[x index=0, y index=1, col sep=comma]  {./Code/MSD/Data/constant/Plot/MSD_linear_proj_error_constant.dat};
				\addplot 
				[color=\ColorLowbound, mark=\MarkerProj, \LineStyleProj, mark size=\MarkerSize, line width=\LineWidth,mark options={solid}]
				table[x index=0, y index=3, col sep=comma]  {./Code/MSD/Data/constant/Plot/MSD_nonlinear_proj_error_lower_bound_constant.dat};
			\end{semilogyaxis}
		\end{tikzpicture}
	}
	
	\subfloat[$u(t) = 0.1\sin (t)$]{
		\begin{tikzpicture}
			\label{fig:MSD_time_output_sin}
			\begin{semilogyaxis}
				[
				width=.35\linewidth,
				height=.20\linewidth, 
				scale only axis,
				grid=both,
				grid style={line width=.1pt, draw=gray!10},
				major grid style={line width=.2pt, draw=gray!50},
				axis lines*=left,
				axis line style={line width=\AxisLineWidth},
				xmin=3.5, xmax=20.5,
				ymin=2e-6, ymax=5e-1,
				ytick={1e-6, 1e-5, 1e-4, 1e-3, 1e-2, 1e-1, 1e0},
				yticklabels = { ,  $10^{-5}$ , , $10^{-3}$, , $10^{-1}$, },
				ylabel= $e_{\outputy}$,
				xlabel= Red. order $r$,
				xlabel style={at={(axis description cs:0.5,-0.1)},anchor=north},
				ylabel style={at={(axis description cs:-0.2,0.5)},anchor=south},
				]
				\addplot 
				[color=\ColorSPG, mark=\MarkerSPG, \LineStyleSPG, mark size=\MarkerSize, line width=\LineWidth,mark options={solid}]
				table[x index=0, y index=1, col sep=comma]  {./Code/MSD/Data/sin/Plot/POD_SPG_output_error.dat};
				\addplot 
				[color=\ColorGMG, mark=\MarkerGMG,\LineStyleGMG, mark size=\MarkerSize, line width=\LineWidth,mark options={solid}]
				table[x index=0, y index=1, col sep=comma]  {./Code/MSD/Data/sin/Plot/POD_GMG_output_error.dat};
				\addplot 
				[color=\ColorGMGQM, mark=\MarkerGMGQM,\LineStyleGMG, mark size=\MarkerSize, line width=\LineWidth,mark options={solid}]
				table[x index=0, y index=3, col sep=comma]  {./Code/MSD/Data/sin/Plot/NL_output_error.dat};
			\end{semilogyaxis}
		\end{tikzpicture}
	}
	\subfloat[$u(t) = 0.1$]{
		\begin{tikzpicture}
			\label{fig:MSD_time_output_cons}
			\begin{semilogyaxis}
				[
				width=.35\linewidth,
				height=.20\linewidth, 
				scale only axis,
				grid=both,
				grid style={line width=.1pt, draw=gray!10},
				major grid style={line width=.2pt, draw=gray!50},
				axis lines*=left,
				axis line style={line width=\AxisLineWidth},
				xmin=3.5, xmax=20.5,
				ymin=2e-6, ymax=5e-1,
				ytick={1e-6, 1e-5, 1e-4, 1e-3, 1e-2, 1e-1, 1e0},
				yticklabels=\empty, 
				xlabel= Red. order $r$,
				xlabel style={at={(axis description cs:0.5,-0.1)},anchor=north},
				]
				\addplot 
				[color=\ColorSPG, mark=\MarkerSPG, \LineStyleSPG, mark size=\MarkerSize, line width=\LineWidth,mark options={solid}]
				table[x index=0, y index=1, col sep=comma]  {./Code/MSD/Data/constant/Plot/POD_SPG_output_error.dat};
				\addplot 
				[color=\ColorGMG, mark=\MarkerGMG,\LineStyleGMG, mark size=\MarkerSize, line width=\LineWidth,mark options={solid}]
				table[x index=0, y index=1, col sep=comma]  {./Code/MSD/Data/constant/Plot/POD_GMG_output_error.dat};
				\addplot 
				[color=\ColorGMGQM, mark=\MarkerGMGQM,\LineStyleGMG, mark size=\MarkerSize, line width=\LineWidth,mark options={solid}]
				table[x index=0, y index=3, col sep=comma]  {./Code/MSD/Data/constant/Plot/NL_output_error.dat};
			\end{semilogyaxis}
		\end{tikzpicture}
	}	
	
	\pgfplotslegendfromname{Legend_Fig_MSD_time}
	\caption{Reduction error $e_{\statex, {\rm red}}$ \eqref{eq:red_error} and output error $e_{\outputy}$ \eqref{eq:output_error} obtained by the \SPLin-\POD-\ROM,  \GMG-\POD-\ROM, and \GMG-\QM-\ROM for the linear mass-spring-damper system with inputs $\inputu(t) = 0.1$ and $\inputu(t)=0.1\sin(t)$. The dotted green lines with square markers denotes the projection error $e_{\statex, {\rm proj}}$ \eqref{eq:proj_error} of corresponding linear approximation maps. The lower bound for the quadratic \ROMs $e_{\statex,{\rm lowerbound}}$ \eqref{equ:non_app_lower_bound} is depicted by a dotted cyan line with square markers.
	}
	\label{fig:MSD_time}
\end{figure}	

\subsection{Nonlinear mass-spring-damper system}
\label{subsec:Num_nonlinear_MSD}

As a second numerical example, we consider a nonlinear mass-spring-damper system consisting of $\FOMsizeHalf$ masses. This problem is based on \cite[Section 5.2.2]{annoni2016modeling}.

For $i=1,\ldots,\FOMsizeHalf-1$, masses $m_i$ and $m_{i+1}$ are connected by a nonlinear spring and a linear damper. 
The corresponding spring and damping coefficients are given by
	$
	k_1+k_2(\xi_{i+1}-\xi_{i})^2$ and $\gamma_i,
	$
	where $\xi_i$ denotes the position of the mass $m_{i}$.
	The last mass $m_n$ is connected to the wall by a nonlinear spring and a linear damper where the spring and damping coefficients are $k_1+k_2\xi_{\FOMsizeHalf}^2$ and $\gamma_{\FOMsizeHalf}$, respectively. Defining the state vector by $\statex=[\statex_1,\ldots,\statex_{\FOMsize}]^{\top}=[\xi,\ldots,\xi_\FOMsizeHalf,\dot{\xi}_1,\ldots,\dot{\xi}_\FOMsizeHalf]^{\top}$, the system admits the \pH representation $\SysLable(\bmJ,\bmR,\curlyH,\bmB,\inputu,\outputy,\statex_{0})$ with $\bmB, \bmJ, \bmR$ and $\curlyH$ given by
	\begin{align*}
		\bmB&=
		\begin{bmatrix}
			\bzero_\FOMsizeHalf\\
			1\\
			\bzero_{\FOMsizeHalf-1}
		\end{bmatrix},\quad
		\bmJ= 
		\begin{bmatrix}
			\bzero_{\FOMsizeHalf,\FOMsizeHalf} & \mathbb{I}_{\FOMsizeHalf}\\
			-\mathbb{I}_{\FOMsizeHalf} & \bzero_{\FOMsizeHalf,\FOMsizeHalf}
		\end{bmatrix},\quad
		\bmR = 
		\begin{bmatrix}
			\bzero_{\FOMsizeHalf,\FOMsizeHalf} & \bzero_{\FOMsizeHalf,\FOMsizeHalf}\\
			\bzero_{\FOMsizeHalf,\FOMsizeHalf} & {\rm diag}(\gamma_1,\ldots, \gamma_\FOMsizeHalf)
		\end{bmatrix},
		\\
		\curlyH(\statex)&= 
		\sum_{i=1}^{\FOMsizeHalf}\frac{1}{2}m\statex_{\FOMsizeHalf+i}^2
		+\sum_{i=1}^{\FOMsizeHalf-1}
		(\frac{1}{2}k_1 (\statex_{i+1}-\statex_{i})^{2} +\frac{1}{4}k_2 (\statex_{i+1}-\statex_{i})^4)
		+
		\frac{1}{2}k_1 \statex_{\FOMsizeHalf}^{2} +\frac{1}{4}k_2 \statex_{\FOMsizeHalf}^4.
	\end{align*}
	We decompose the Hamiltonian into a quadratic and a nonlinear part, 
	\begin{align*}\curlyH(\statex)=\frac{1}{2}\statex^{\top}\bmQ\statex+
	\DEIMsplitHam
	(\statex),
	\end{align*}
	where 
	\begin{align*}
		\bmQ=
		\begin{bmatrix}
			\bmQ_0&\bzero_{\FOMsizeHalf,\FOMsizeHalf}\\
			\bzero_{\FOMsizeHalf,\FOMsizeHalf}& \mathbb{I}_{\FOMsizeHalf}
		\end{bmatrix},\quad \text{with} \quad 
		\bmQ_0=
		\begin{bmatrix}
			1 & -1& 0   & \ldots & 0 & 0\\
			-1 & 2& -1 &            &    & 0\\
			\vdots & &\ddots  & \ddots      &  & \\
			0 & 0& \ldots & -1 & 2 & -1\\
			0 & 0& \ldots & 0 & -1 & 2
		\end{bmatrix}.
	\end{align*}
	The nonlinear part is given by $\DEIMsplitHam(x) = \sum_{k=1}^{\FOMsizeHalf-1}\phi(\statex_{k}-\statex_{k+1}) +\phi(\statex_{\FOMsizeHalf})$ with
	$
	\phi(x) = \frac{1}{4}x^4.
	$
	To simplify the representation of the nonlinear part, we introduce the transformed state $\widehat{\statex}
		=
		\begin{bmatrix}
			\ell_1,\ldots,\ell_{\FOMsizeHalf},
			\statex_{\FOMsizeHalf+1},
			\ldots,
			\statex_{\FOMsize}
		\end{bmatrix}^{\top}$, where $\ell_i = \statex_{i} - \statex_{i+1},  i=1,\ldots,\FOMsizeHalf-1$, and $\ell_{\FOMsizeHalf} =\statex_{\FOMsizeHalf}$.
	In these coordinates, the nonlinear part reads $\widetilde{p}(\widehat{\statex}) =\sum_{k=1}^{\FOMsizeHalf}\phi(\ell_k)$. 
	Furthermore, let $\widetilde{\bm{T}} \in \R^{\FOMsizeHalf \times \FOMsizeHalf}$ denote the matrix with diagonal entries equal to $1$ and superdiagonal entries equal to $-1$. We define 
	\begin{align*}
		\bm{T} \coloneq \begin{bmatrix} \widetilde{\bm{T}} &\bm{0}_{\FOMsizeHalf,\FOMsizeHalf} \\
		\bm{0}_{\FOMsizeHalf,\FOMsizeHalf} &\mathbb{I}_{\FOMsizeHalf}
		\end{bmatrix},
	\end{align*}
	and write the coordinate transformation in compact form $\widehat{\statex} = \bm{T} \statex$. In the transformed coordinates, the \pH system is given by $\SysLable(\newbmJ,\newbmR,\widetilde{\curlyH},\newbmB,\inputu,\outputy,\widehat{\statex}_{0})$, where
	\begin{align*}
		\newbmJ = \bm{T} \bmJ \bm{T}^{\top}, \quad \newbmR =  \bm{T} \bmR \bm{T}^{\top}, \quad \newbmQ = \bm{T}^{-T} \bmQ \bm{T}^{-1}, \quad \newbmB = \bm{T} \bmB, 
	\end{align*}
	as well as 
	\begin{align*}
		\widetilde{\curlyH}( \widehat{\statex}) = \frac{1}{2} \widehat{\statex}^{\top}\newbmQ\widehat{\statex} +\widetilde{\DEIMsplitHam}(\widehat{\statex}), \quad \text{and} \quad 	\widehat{\statex}_{0} = \bm{T} \statex_0. 
	\end{align*}
	
	In this example, we consider the nonlinear mass-spring-damper system with $\FOMsizeHalf = 500$ resulting in a system of $\FOMsize = 1000$. The system is simulated on the time interval $(0,100s]$, with zero initial condition, i.e., $\statex_0 = \bzero$. The remaining parameters are set to $k_1= k_2 =1$, $m_j\ \mathrm{and}\ \gamma_j = 0.3$ for $j = 1,\ldots, \FOMsizeHalf$. We use a uniform time discretization with step size $\Delta t =10^{-1}s$, resulting in $n_t=1001$ time instances. Further, we consider the two different inputs $u(t) =0.1$ and $u(t) = 0.1\sin(t)$.
	Since we consider nonlinear \pH systems, we use the \DEIM method for \pH systems introduced in \cite{chaturantabut2016structure} with tolerance $\epsilon_{\DEIM}= 10^{-8}$.\footnote{The authors of \cite{pagliantini2023gradient} proposed a gradient-preserving adaptive DEIM, which is different from the structure-preserving \DEIM that we use. However, for the nonlinear mass-spring-damper system we consider here, these two \DEIM approximations are equivalent.}
	We compare the \GMG-\POD-\ROM, \GMG-\QM-\ROM and \SPQuad-\POD-\ROM for reduced orders $\ROMsize$ from $6$ to $20$. The tolerance and maximum number of iterations of Newton's method is set to $10^{-8}$ and $20$ respectively. For the \GMG-\QM-\ROM, the regularization parameter is selected empirically according to
	$
	\lambda_{\rm reg} = \max\left(0.2*e_{\rm proj}(r),10^{-2.5}\right).
	$
	\begin{figure}[htbp]
		\centering 
		\subfloat[$u(t) = 0.1\sin (t)$]{
			\begin{tikzpicture}
				\label{fig:NL_MSD_time_state_sin}
				\begin{semilogyaxis}
					[
					width=.35\linewidth,
					height=.20\linewidth, 
					scale only axis,
					grid=both,
					grid style={line width=.1pt, draw=gray!10},
					major grid style={line width=.2pt, draw=gray!50},
					axis lines*=left,
					axis line style={line width=\AxisLineWidth},
					xmin=5.5, xmax=20.5,
					ymin=3e-4, ymax=3e-1,
					ytick={1e-4, 1e-3, 1e-2, 1e-1, 1e-0},
					yticklabels = {, $10^{-3}$, $10^{-2}$ , $10^{-1}$, },
					ylabel= $e_{\statex}$,
					xtick={8, 12, 16,20},
					ylabel style={at={(axis description cs:-0.15,0.5)},anchor=south},
					xlabel= Red. order $r$,
					xlabel style={at={(axis description cs:0.5,-0.1)},anchor=north},
					legend columns=3,
					legend entries={\SPQuad-\POD-\ROM, \GMG-\POD-\ROM, \GMG-\QM-\ROM, $e_{\statex, {\rm proj}}$, $e_{\statex,{\rm lowerbound}}$},
					legend to name= Legend_Fig_NLMSD_time,
					unbounded coords=jump, 
					]
					\addplot 
					[color=\ColorSP, mark=\MarkerSP, \LineStyleSP, mark size=\MarkerSize, line width=\LineWidth,mark options={solid}]
					table[x index=0, y index=1, col sep=comma] 
					{./Code/NL_MSD/Data/sin/Plot/Gugercin_state_error.dat};
					\addplot 
					[color=\ColorGMG, mark=\MarkerGMG,\LineStyleGMG, mark size=\MarkerSize, line width=\LineWidth,mark options={solid}]
					table[x index=0, y index=1, col sep=comma]  
					{./Code/NL_MSD/Data/sin/Plot/GMG_POD_state_error.dat};
					\addplot 
					[color=\ColorGMGQM, mark=\MarkerGMGQM, \LineStyleGMG, mark size=\MarkerSize, line width=\LineWidth,mark options={solid}]
					table[x index=0, y index=3, col sep=comma]  
					{./Code/NL_MSD/Data/sin/Plot/NL_state_error.dat};
					\addplot 
					[color=\ColorGMG, mark=\MarkerProj, \LineStyleProj, mark size=\MarkerSize, line width=\LineWidth,mark options={solid}]
					table[x index=0, y index=1, col sep=comma]  {./Code/NL_MSD/Data/sin/Plot/NLMSD_linear_proj_error_sin.dat};
					\addplot 
					[color =\ColorLowbound, mark=\MarkerProj, \LineStyleProj,  mark size=\MarkerSize, line width=\LineWidth,mark options={solid}]
					table[x index=0, y index=3, col sep=comma]  {./Code/NL_MSD/Data/sin/Plot/NLMSD_nonlinear_proj_lower_bound_sin.dat};
				\end{semilogyaxis}
			\end{tikzpicture}
		}
		\subfloat[$u(t) = 0.1$]{
			\begin{tikzpicture}
				\label{fig:NL_MSD_time_state_cons}
				\begin{semilogyaxis}
					[
					width=.35\linewidth,
					height=.20\linewidth, 
					scale only axis,
					grid=both,
					grid style={line width=.1pt, draw=gray!10},
					major grid style={line width=.2pt, draw=gray!50},
					axis lines*=left,
					axis line style={line width=\AxisLineWidth},
					xmin=5.5, xmax=20.5,
					ymin=3e-4, ymax=3e-1,
					ytick={1e-4, 1e-3, 1e-2, 1e-1, 1e-0},
					yticklabels=\empty, 
					xtick={8, 12, 16,20},
					xlabel= Red. order $r$,
					xlabel style={at={(axis description cs:0.5,-0.1)},anchor=north},
					]
					\addplot 
					[color=\ColorSP, mark=\MarkerSP, \LineStyleSP, mark size=\MarkerSize, line width=\LineWidth,mark options={solid}]
					table[x index=0, y index=1, col sep=comma]  {./Code/NL_MSD/Data/constant/Plot/Gugercin_state_error.dat};
					\addplot 
					[color=\ColorGMG, mark=\MarkerGMG,\LineStyleGMG, mark size=\MarkerSize, line width=\LineWidth,mark options={solid}]
					table[x index=0, y index=1, col sep=comma]  {./Code/NL_MSD/Data/constant/Plot/GMG_POD_state_error.dat};
					\addplot 
					[color=\ColorGMGQM, mark=\MarkerGMGQM, \LineStyleGMG, mark size=\MarkerSize, line width=\LineWidth,mark options={solid}]
					table[x index=0, y index=3, col sep=comma]  
					{./Code/NL_MSD/Data/constant/Plot/NL_state_error.dat};
					\addplot 
					[color=\ColorGMG, mark=\MarkerProj, \LineStyleProj, mark size=\MarkerSize, line width=\LineWidth,mark options={solid}]
					table[x index=0, y index=1, col sep=comma]  {./Code/NL_MSD/Data/constant/Plot/NLMSD_linear_proj_error_constant.dat};
					\addplot 
					[color =\ColorLowbound, mark=\MarkerProj, \LineStyleProj, mark size=\MarkerSize, line width=\LineWidth,mark options={solid}]
					table[x index=0, y index=3, col sep=comma]  {./Code/NL_MSD/Data/constant/Plot/NLMSD_nonlinear_proj_lower_bound_constant.dat};
				\end{semilogyaxis}
			\end{tikzpicture}
		}
		
		\subfloat[$u(t) = 0.1\sin (t)$]{
			\begin{tikzpicture}
				\label{fig:NL_MSD_time_output_sin}
				\begin{semilogyaxis}
					[
					width=.35\linewidth,
					height=.20\linewidth, 
					scale only axis,
					grid=both,
					grid style={line width=.1pt, draw=gray!10},
					major grid style={line width=.2pt, draw=gray!50},
					axis lines*=left,
					axis line style={line width=\AxisLineWidth},
					xmin=5.5, xmax=20.5,
					ymin=3e-4, ymax=3e-1,
					ytick={1e-4, 1e-3, 1e-2, 1e-1, 1e-0},
					yticklabels = {, $10^{-3}$, $10^{-2}$ , $10^{-1}$, },
					xtick={8, 12, 16,20},
					ylabel= $e_{\outputy}$,
					xlabel= Red. order $r$,
					xlabel style={at={(axis description cs:0.5,-0.1)},anchor=north},
					ylabel style={at={(axis description cs:-0.15,0.5)},anchor=south},
					]
					\addplot 
					[color=\ColorSP, mark=\MarkerSP, \LineStyleSP, mark size=\MarkerSize, line width=\LineWidth,mark options={solid}]
					table[x index=0, y index=1, col sep=comma] 
					{./Code/NL_MSD/Data/sin/Plot/Gugercin_output_error.dat};
					\addplot 
					[color=\ColorGMG, mark=\MarkerGMG,\LineStyleGMG, mark size=\MarkerSize, line width=\LineWidth,mark options={solid}]
					table[x index=0, y index=1, col sep=comma] 
					{./Code/NL_MSD/Data/sin/Plot/GMG_POD_output_error.dat};
					\addplot 
					[color=\ColorGMGQM, mark=\MarkerGMGQM, \LineStyleGMG, mark size=\MarkerSize, line width=\LineWidth,mark options={solid}]
					table[x index=0, y index=3, col sep=comma]  
					{./Code/NL_MSD/Data/sin/Plot/NL_output_error.dat};
				\end{semilogyaxis}
			\end{tikzpicture}
		}
		\subfloat[$u(t) = 0.1$]{
			\begin{tikzpicture}
				\label{fig:NL_MSD_time_output_cons}
				\begin{semilogyaxis}
					[
					width=.35\linewidth,
					height=.20\linewidth, 
					scale only axis,
					grid=both,
					grid style={line width=.1pt, draw=gray!10},
					major grid style={line width=.2pt, draw=gray!50},
					axis lines*=left,
					axis line style={line width=\AxisLineWidth},
					xmin=5.5, xmax=20.5,
					ymin=3e-4, ymax=3e-1,
					ytick={1e-4, 1e-3, 1e-2, 1e-1, 1e-0},
					yticklabels=\empty, 
					xtick={8, 12, 16,20},
					xlabel= Red. order $r$,
					xlabel style={at={(axis description cs:0.5,-0.1)},anchor=north},
					]
					\addplot 
					[color=\ColorSP, mark=\MarkerSP, \LineStyleSP, mark size=\MarkerSize, line width=\LineWidth,mark options={solid}]
					table[x index=0, y index=1, col sep=comma]  {./Code/NL_MSD/Data/constant/Plot/Gugercin_output_error.dat};
					\addplot 
					[color=\ColorGMG, mark=\MarkerGMG,\LineStyleGMG, mark size=\MarkerSize, line width=\LineWidth,mark options={solid}]
					table[x index=0, y index=1, col sep=comma]  {./Code/NL_MSD/Data/constant/Plot/GMG_POD_output_error.dat};
					\addplot 
					[color=\ColorGMGQM, mark=\MarkerGMGQM, \LineStyleGMG, mark size=\MarkerSize, line width=\LineWidth,mark options={solid}]
					table[x index=0, y index=3, col sep=comma]  
					{./Code/NL_MSD/Data/constant/Plot/NL_output_error.dat};
				\end{semilogyaxis}
			\end{tikzpicture}
		}
		
		\pgfplotslegendfromname{Legend_Fig_NLMSD_time}
	\caption{Reduction error $e_{\statex, {\rm red}}$ \eqref{eq:red_error} and output error $e_{\outputy}$ \eqref{eq:output_error} obtained by the \SPQuad-\POD-\ROM,  \GMG-\POD-\ROM, and \GMG-\QM-\ROM for the nonlinear mass-spring-damper system with inputs $\inputu(t) = 0.1$ and $\inputu(t)=0.1\sin(t)$. The dotted green lines with square markers denotes the projection error $e_{\statex, {\rm proj}}$ \eqref{eq:proj_error} of corresponding linear approximation maps. The lower bound for the quadratic \ROMs $e_{\statex,{\rm lowerbound}}$ \eqref{equ:non_app_lower_bound} is depicted by a dotted cyan line with square markers.
	}
		\label{fig:NL_MSD_time}
	\end{figure}

	\begin{figure}[htbp]
		\centering 
		\subfloat[$u(t) = 0.1\sin(t)$]{
			\begin{tikzpicture}
				\begin{semilogyaxis}
					[
					width=.35\linewidth,
					height=.20\linewidth, 
					scale only axis,
					grid=both,
					grid style={line width=.1pt, draw=gray!10},
					major grid style={line width=.2pt, draw=gray!50},
					axis lines*=left,
					axis line style={line width=\AxisLineWidth},
					xmin=0, xmax=100,
					ymin=1e-8, ymax=1e-4,
					xtick={0, 20, 40, 60, 80, 100},
					ytick={1e-8, 1e-7, 1e-6, 1e-5, 1e-4},
					yticklabels = {, $10^{-7}$,  , $10^{-5}$, },
					ylabel= $e_{\rm energy}(t_{\text{end}})$,
					xlabel= Time (s),
					xlabel style={at={(axis description cs:0.5,-0.1)},anchor=north},
					ylabel style={at={(axis description cs:-0.15,0.5)},anchor=south},
					legend columns=4,
					legend entries={\SPQuad-\POD-\ROM, \GMG-\POD-\ROM,\GMG-\QM-\ROM,\FOM},
					legend to name= Legend_Fig_NLMSD_energy,
					unbounded coords=jump, 
					]
					\addplot 
					[color=\ColorSP, solid,  line width=\LineWidth, no markers]
					table[x index=0, y index=6, col sep=comma]  {./Code/NL_MSD/Data/sin/Plot/Gugercin_energy_dif.dat};
					\addplot 
					[color = \ColorGMG, dashed, green, line width=\LineWidth,no markers]
					table[x index=0, y index=6, col sep=comma]  {./Code/NL_MSD/Data/sin/Plot/GMG_POD_energy_dif.dat};
					\addplot 
					[color = \ColorGMGQM,  loosely dashed, line width=\LineWidth,no markers]
					table[x index=0, y index=6, col sep=comma]  {./Code/NL_MSD/Data/sin/Plot/NL_energy_dif_r_q_3.dat};
					\addplot 
					[color = \ColorFOM, dotted, line width=\LineWidth,no markers]
					table[x index=0, y index=1, col sep=comma]  {./Code/NL_MSD/Data/sin/Plot/FOM_energy_dif.dat};
				\end{semilogyaxis}
			\end{tikzpicture}
		}
		\subfloat[$u(t) = 0.1$]{
			\begin{tikzpicture}
				\begin{semilogyaxis}
					[
					width=.35\linewidth,
					height=.20\linewidth, 
					scale only axis,
					grid=both,
					grid style={line width=.1pt, draw=gray!10},
					major grid style={line width=.2pt, draw=gray!50},
					axis lines*=left,
					axis line style={line width=\AxisLineWidth},
					xmin=0, xmax=100,
					ymin=1e-8, ymax=1e-4,
					xtick={0, 20, 40, 60, 80, 100},
					ytick={1e-8, 1e-7, 1e-6, 1e-5, 1e-4},
					yticklabels= \empty,
					xlabel= Time (s),
					xlabel style={at={(axis description cs:0.5,-0.1)},anchor=north},
					ylabel style={at={(axis description cs:-0.15,0.5)},anchor=south},
					]
					\addplot 
					[color=\ColorSP,  solid,  line width=\LineWidth, no markers]
					table[x index=0, y index=6, col sep=comma]  {./Code/NL_MSD/Data/constant/Plot/Gugercin_energy_dif.dat};
					\addplot 
					[color = \ColorGMG, dashed, green, line width=\LineWidth,no markers]
					table[x index=0, y index=6, col sep=comma]  {./Code/NL_MSD/Data/constant/Plot/GMG_POD_energy_dif.dat};
					\addplot 
					[color = \ColorGMGQM,  loosely dashed, line width=\LineWidth,no markers]
					table[x index=0, y index=6, col sep=comma]  {./Code/NL_MSD/Data/constant/Plot/NL_energy_dif_r_q_3.dat};
					\addplot 
					[color = \ColorFOM, dotted, line width=\LineWidth,no markers]
					table[x index=0, y index=1, col sep=comma] 
					{./Code/NL_MSD/Data/constant/Plot/FOM_energy_dif.dat};
				\end{semilogyaxis}
			\end{tikzpicture}
		}
		
		\pgfplotslegendfromname{Legend_Fig_NLMSD_energy}
		\caption{Error in the energy balance equation \eqref{equ:def_energy_conserve_error} of the \FOM and \ROMs with a reduced order of 16 obtained by the \SPQuad-\POD-\ROM, \GMG-\POD-\ROM, and \GMG-\QM-\ROM for the nonlinear mass-spring-damper system with two different inputs.}
		\label{fig:NL_MSD_energy_error}
	\end{figure}	

	Figure \ref{fig:NL_MSD_time} shows that the \SPQuad-\POD-\ROM and \GMG-\POD-\ROM have close performance with respect of the state error, while the \GMG-\POD-\ROM achieves a smaller output error than the \SPQuad-\POD-\ROM. By considering a quadratic embedding, the resulting \GMG-\QM-\ROM further improves both, the state and the output error. 
	Regarding the error in the energy balance equation \eqref{equ:energy_conserve_time_int}, Figure \ref{fig:NL_MSD_energy_error} shows that all considered \ROMs attain a level of accuracy comparable to that of the \FOM. 
	
\section{Conclusions}
	\label{Sec:Outlook}

We propose a new structure-preserving \MOR \framework for \pH systems with general nonlinear approximation maps based on the \GMG projection. 
We established sufficient conditions for structure preservation, showed that the associated non-degeneracy conditions are generically satisfied, and provided two concrete realizations with a linear and a quadratic approximation map for linear and nonlinear \pH systems
Numerical examples of a linear and a nonlinear mass-spring-damper system are investigated, and the \GMG-\POD-\ROM and the \GMG-\QM-\ROM provide a lower relative error compared to known methods in the literature in terms of state and output. 
With respect to future work, we aim to consider different classes of \pH systems, e.g., a \pH system, where the resistive term is a nonlinear function of $\DivHamFOM$ or employing structure-preserving neural networks to speed up the evaluation of a state-dependent systems matrices in a reduced order \pH model.  
	
\appendix 
\section{Full proof of Theorem~\ref{Theo:exists_JR_inv}} \label{sec_app_proofs}

Before including the full proof of Theorem~\ref{Theo:exists_JR_inv}, we state one lemma.
	\begin{lemma}
		\label{lemma:ker_JR}
		Let $\bmJ=-\bmJ^{\top}\in\R^{\FOMsize\times \FOMsize}$ be a skew-symmetric matrix and let $\bmR=\bmR^{\top}\in\R^{\FOMsize\times\FOMsize}$ be a positive semi-definite matrix. Then
		\begin{equation*}
			\mathrm{det}\left(\bmJ-\bmR\right)\neq 0 
			\Leftrightarrow
			\mathrm{ker}\left(\bmJ\right)\cap \mathrm{ker}\left(\bmR\right)=\{\bzero\}.
		\end{equation*} 
	\end{lemma}
	\begin{proof}
	We start by showing that $\mathrm{det}\left(\bmJ-\bmR\right)\neq 0 
		\Rightarrow
		\mathrm{ker}\left(\bmJ\right)\cap \mathrm{ker}\left(\bmR\right)=\{\bzero\}$, which we prove by contraposition. 
	If $\mathrm{ker}\left(\bmJ\right)\cap \mathrm{ker}\left(\bmR\right)\neq\{\bzero\}$, then there exist $\statex\in\mathrm{ker}\left(\bmJ\right)\cap \mathrm{ker}\left(\bmR\right)$ s.t. $\statex\neq 0$ elementwise and 
			$(\bmJ-\bmR)\statex =\bzero$. As a result we get that $\mathrm{det}\left(\bmJ-\bmR\right)= 0$. 
			
			For the converse direction, we assume that $\mathrm{ker}\left(\bmJ\right)\cap \mathrm{ker}\left(\bmR\right)=\{\bzero\}$ and we consider an arbitrary non-zero vector $\statex\in\R^{\FOMsize}$. 
			If $\statex\in\mathrm{ker}(\bmR)$, then $\bmJ\statex\neq\bzero$. Thus, $(\bmJ-\bmR)\statex 
			=\bmJ\statex
			\neq\bzero$. If $\statex\notin\mathrm{ker}(\bmR)$, then $\statex^{\top}(\bmJ-\bmR)\statex = -\statex^{\top}\bmR\statex < 0$, which results to $(\bmJ-\bmR)\statex \neq\bzero$.
			Consequently,  $\mathrm{det}\left(\bmJ-\bmR\right)\neq 0 $.

	\end{proof}

\begin{proof}[Proof of Theorem~\ref{Theo:exists_JR_inv}]
			For any $1\le\ROMsize\leq \FOMsize$, let $f_{\mathrm{det},\ROMsize,1}\colon\R^{\FOMsize\times\ROMsize}\rightarrow \R$ be defined by 
			\begin{align*}
				f_{\mathrm{det},\ROMsize,1}\colon\bmU\mapsto \allowbreak \mathrm{det}\left(\bmU^{\top}\left(\Jcheck-\Rcheck\right)\bmU\right).
			\end{align*}
			Note that $f_{\mathrm{det},\ROMsize,1}$ is a polynomial with respect to the entries of $\bmU$. This, means that $f_{\mathrm{det},\ROMsize,1}$ is a real analytic function.
			By Lemma~\ref{lemma:analytic_measure}, it is sufficient to show that $f_{\mathrm{det},\ROMsize,1}$ is not identically zero. Indeed, if $f_{\mathrm{det},\ROMsize,1} \not\equiv 0$, then its zero set has Lebesgue measure zero. Therefore 
			\begin{equation*}
				\mathrm{det}\left(\bmU^{\top}\left(\Jcheck-\Rcheck\right)\bmU\right)\neq 0,
			\end{equation*}
			holds for almost every $ \bmU\in\R^{\FOMsize\times \ROMsize} $ for $\ROMsize=1,\ldots,\FOMsize$. 
			In the following, we show, that such $\bmU$ exists. To this end, we consider the two different cases $\mathrm{det}(\Jcheck)=0$, $\mathrm{det}(\Jcheck)\neq0$ and provide for each case an explicit construction of $\bmU$. 
			
We first consider the case $\mathrm{det}(\Jcheck)=0$. By \cite[Theorem 2.2]{benner2000cholesky} there exists an invertible matrix
					\begin{equation*}
						\bmV=
						\begin{bmatrix}
							\bmV_1,\bmV_2
						\end{bmatrix}\in\R^{\FOMsize\times \FOMsize},\
						\mathrm{with}\
						\bmV_1 =[\bmv_1,\ldots,\bmv_{2\ell}]
						\in\R^{\FOMsize\times 2\ell}\
						,\
						\bmV_2\in\R^{\FOMsize\times\ (\FOMsize-2\ell)}
						\
						(2\ell<\FOMsize),
					\end{equation*}
					such that $\Jcheck \bmV_2 = \bzero$ and 
					\begin{equation*}
						\label{equ:GMG_exists_proof_SR_decomp}
						\bmV^{\top}\Jcheck\bmV = \mathrm{diag}(\jtwo,\ldots,\jtwo,\underbrace{0\ldots,0}_{\FOMsize-2\ell}),
					\end{equation*}
					where $\jtwo =\begin{pmatrix}
						0&-1\\
						1&0
					\end{pmatrix}$. 
From $\Jcheck \bmV_2 = \bzero$, we know that $\mathrm{span}(\bmV_2)\subseteq \mathrm{ker}\left(\Jcheck\right)$. Since $\mathrm{det}(\Jcheck-\Rcheck)\neq 0$, 
					we follow from Lemma \ref{lemma:ker_JR} that $\mathrm{span}(\bmV_2)\cap \ker(\Rcheck)=\{\bzero\}$. 
				 	Further, the matrix $\bmV_2^{\top}\Rcheck\bmV_2$ is symmetric positive definite matrix. Then, there exists an invertible matrix $\bmK\in\R^{(\FOMsize-2\ell)\times (\FOMsize-2\ell)}$ s.t.
					\begin{align*}
						\bmK^{\top}\bmV_2^{\top}\Rcheck\bmV_2\bmK \allowbreak= \mathrm{diag}
						(\lambda_1,\ldots,\lambda_{\FOMsize-2\ell}),
						\ \mathrm{where}\
						\lambda_{i}>0, i=1,\ldots, \FOMsize-2\ell.
					\end{align*}
					Let $\bmW := \bmV_2\bmK$ and denote its columns by $\{\bmw_1,\ldots,\bmw_{\FOMsize-2\ell}\}$. For a fixed $r \in \{1, \dots, \FOMsize\}$, we choose integers $i,j$ such that 					$0\leq i\leq \ell, 0\leq j\leq \FOMsize-2\ell$ and $r = 2i+j$. We set $\bmU$ to contain the first $2i$ columns of $\bmV_1$ and the $j$ first columns of $\bmW$, i.e., 
					\begin{align*}\bmU \coloneqq 
					\begin{bmatrix}
						\bmv_1,\ldots,\bmv_{2i},\bmw_1,\ldots,\bmw_j
					\end{bmatrix},
					\end{align*}
					and further define $\bmU_2 := \begin{bmatrix} \bmw_1,\ldots,\bmw_j \end{bmatrix}$. By construction, we arrive at 
					\begin{equation*}
						\bmU^{\top}\Jcheck\bmU=
						\mathrm{diag}(\underbrace{\jtwo,\ldots,\jtwo}_{i},\underbrace{0\ldots,0}_{j}),\
						\mathrm{and}\
						\bmU_2^{\top}\Rcheck\bmU_2=
						\mathrm{diag}
						\left(\lambda_1,\ldots,\lambda_j\right).
					\end{equation*}
					We then arrive at $\mathrm{ker}\left(\bmU^{\top}\Jcheck\bmU\right)\cap \mathrm{ker}\left(\bmU^{\top}\Rcheck\bmU\right) =\{\bzero\}.$
					By Lemma~\ref{lemma:ker_JR} we get that\\ $\mathrm{det}\left(\bmU^{\top}\left(\Jcheck-\Rcheck\right)\bmU\right)\neq 0$.\\
					We now consider the case that $\mathrm{det}(\Jcheck)\neq 0.$ Again, by \cite[Theorem 2.2]{benner2000cholesky}, there exists an invertible matrix
						$\bmV=[\bmv_1,\ldots,\bmv_{\FOMsize}] \in \mathbb{R}^{\FOMsize \times \FOMsize}$
					such that
					\begin{equation*}
						\label{equ:GMG_exists_proof_SR_decomp}
						\bmV^{\top}\Jcheck\bmV = \mathrm{diag}(\jtwo,\ldots,\jtwo)\in\R^{\FOMsize\times \FOMsize}.
					\end{equation*}
                If $\bmR\neq \bzero$, then there exists $\statex\in\R^{\FOMsize}$, s.t.
                     \begin{align*}
                     	0>
                     	-\statex^{\top}\bmR\statex
                     	=
                     	\statex^{\top}(\bmJ-\bmR)^{\top}(\Jcheck-\Rcheck)(\bmJ-\bmR)\statex
                     	=
                     	-
                     	\statex^{\top}(\bmJ-\bmR)^{\top}\Rcheck(\bmJ-\bmR)\statex.
                     \end{align*}
                     Thus,  $\Rcheck\neq \bzero$ and $\Rcheck\succeq 0$.
                     Since $\bmV$ is invertible, there exists $\bmv\in\{\bmv_i\}_{i=1}^{\FOMsize}$ s.t. $\bmv^{\top}\Rcheck \bmv>0$. 
					Without loss of generality, we assume that $\bmv_{\FOMsize}^{\top}\Rcheck \bmv_{\FOMsize}>0$. If $\ROMsize$ is an \textbf{even} number we construct $\bmU$ by $\bmU=\begin{bmatrix} \bmv_1,\ldots,\bmv_{\ROMsize} \end{bmatrix}$. Then $\mathrm{ker}(\bmU^{\top}\Jcheck\bmU)=  \{\bzero\}$ and by Lemma~\ref{lemma:ker_JR}, we get $\mathrm{det}(\bmU^{\top}\left(\Jcheck-\Rcheck\right)\bmU)\neq 0$. \\
					If $\ROMsize$ is an \textbf{odd} number, we construct $\bmU$ by $\bmU\coloneqq
					\begin{bmatrix}
						\bmv_1,\ldots,\bmv_{\ROMsize-1},\bmv_{\FOMsize}
					\end{bmatrix}$
					Since
					\begin{equation*}
						\bmU^{\top}\Jcheck\bmU=
						\mathrm{diag}(\underbrace{\jtwo,\ldots,\jtwo}_{(r-1)/2},0),\
						\mathrm{and}\
						v_{\FOMsize}^{\top}\Rcheck v_{\FOMsize}>0,
					\end{equation*}
					we have 
					$\mathrm{ker}\left(\bmU^{\top}\Jcheck\bmU\right)\cap
					\mathrm{ker}\left(\bmU^{\top}\Rcheck\bmU\right)
					=\{\bzero\}.$ 
					By Lemma~\ref{lemma:ker_JR}, $\mathrm{det}\left(\bmU^{\top}\left(\Jcheck-\Rcheck\right)\bmU\right)\neq 0$ holds.	
\end{proof}

\bibliographystyle{abbrvurl}
\bibliography{References}

@article{kotyczka2018discrete,
	title={Discrete-time port-{H}amiltonian systems based on {G}auss-{L}egendre collocation},
	author={Kotyczka, Paul and Lefevre, Laurent},
	journal={IFAC-PapersOnLine},
	volume={51},
	number={3},
	pages={125--130},
	year={2018},
	publisher={Elsevier}
}

@article{BennerGugercinWilcox2015,
author = {Benner, Peter and Gugercin, Serkan and Willcox, Karen},
title = {A Survey of Projection-Based Model Reduction Methods for Parametric Dynamical Systems},
journal = {SIAM Review},
volume = {57},
number = {4},
pages = {483-531},
year = {2015},
doi = {10.1137/130932715}
}

@book{BennerOhlbergerCW2017,
author = {Benner, Peter and Ohlberger, Mario and Cohen, Albert and Willcox, Karen},
title = {Model Reduction and Approximation},
publisher = {Society for Industrial and Applied Mathematics},
year = {2017},
doi = {10.1137/1.9781611974829},
address = {Philadelphia, PA},
edition   = {}
}

@article{Hesthaven_Peherstorfer_Unger_2026, 
	title={Nonlinear model reduction for transport-dominated problems}, 
	volume={35}, 
	DOI={10.1017/S0962492926100294}, 
	journal={Acta Numerica}, 
	author={Hesthaven, Jan S. and Peherstorfer, Benjamin and Unger, Benjamin}, 
	year={2026}, 
	pages={173–272}
}

@article{Rettberg31122023,
author = {Johannes Rettberg and Dominik Wittwar and Patrick Buchfink and Alexander Brauchler and Pascal Ziegler and Jörg Fehr and Bernard Haasdonk},
title = {Port-{H}amiltonian fluid–structure interaction modelling and structure-preserving model order reduction of a classical guitar},
journal = {Mathematical and Computer Modelling of Dynamical Systems},
volume = {29},
number = {1},
pages = {116--148},
year = {2023},
publisher = {Taylor \& Francis},
doi = {10.1080/13873954.2023.2173238}
}

@article{BUCHFINK2026109784,
title = {Energy-stable port-{H}amiltonian systems},
journal = {Applied Mathematics Letters},
volume = {173},
pages = {109784},
year = {2026},
issn = {0893-9659},
doi = {https://doi.org/10.1016/j.aml.2025.109784},
author = {Patrick Buchfink and Silke Glas and Hans Zwart}
}

@article{GruberTezaur2025,
author = {Gruber, Anthony and Tezaur, Irina},
title = {{Variationally Consistent Hamiltonian Model Reduction}},
journal = {SIAM Journal on Applied Dynamical Systems},
volume = {24},
number = {1},
pages = {376-414},
year = {2025},
doi = {10.1137/24M1652490}
}

@book{Antoulas2005,
author = {Antoulas, Athanasios C.},
title = {Approximation of Large-Scale Dynamical Systems},
publisher = {Society for Industrial and Applied Mathematics},
year = {2005},
doi = {10.1137/1.9780898718713},
address = {},
edition   = {}
}

@article{mehrmann2023control,
	title={Control of port-{H}amiltonian differential-algebraic systems and applications},
	author={Mehrmann, Volker and Unger, Benjamin},
	journal={Acta Numerica},
	volume={32},
	pages={395--515},
	year={2023},
	publisher={Cambridge University Press}
}

@article{Mityagin2020,
	author = {B. S. Mityagin},
	title = {The Zero Set of a Real Analytic Function},
	journal = {Mathematical Notes},
	volume = {107},
	pages = {529--530},
	year = {2020},
	doi = {10.1134/S0001434620030189}
}

@article{buchfink2024approximation,
	author = {Patrick Buchfink and Silke Glas and Bernard Haasdonk},
	doi = {10.5802/crmath.632},
	issn = {1778-3569},
	issue = {G13},
	journal = {Comptes Rendus. Mathématique},
	month = {12},
	pages = {1881-1891},
	title = {Approximation Bounds for Model Reduction on Polynomially Mapped Manifolds},
	volume = {362},
	year = {2024}
}

@article{benner2000cholesky,
	author = {Peter Benner and Ralph Byers and Heike Fassbender and Volker Mehrmann and David Watkins},
	issn = {1068-9613},
	journal = {Electronic Transactions on Numerical Analysis (ETNA)},
	pages = {85-93},
	title = {Cholesky-like factorizations of skew-symmetric matrices},
	volume = {11},
	year = {2000}
}

@article{rettberg2024data,
	author = {Johannes Rettberg and Jonas Kneifl and Julius Herb and Patrick Buchfink and Jörg Fehr and Bernard Haasdonk},
	month = {8},
	title = {Data-driven identification of latent port-{H}amiltonian systems},
	year = {2024},
	journal={arXiv preprint arXiv:2408.08185}
}

@article{geng2025data,
	author = {Yuwei Geng and Lili Ju and Boris Kramer and Zhu Wang},
	doi = {10.1016/j.cma.2025.118042},
	issn = {00457825},
	journal = {Computer Methods in Applied Mechanics and Engineering},
	month = {7},
	pages = {118042},
	publisher = {Elsevier B.V.},
	title = {Data-driven reduced-order models for port-{H}amiltonian systems with operator inference},
	volume = {442},
	year = {2025}
}

@article{greif2019decay,
	author = {Constantin Greif and Karsten Urban},
	doi = {10.1016/j.aml.2019.05.013},
	issn = {18735452},
	journal = {Applied Mathematics Letters},
	pages = {216-222},
	publisher = {Elsevier Ltd},
	title = {Decay of the {K}olmogorov {N}-width for wave problems},
	volume = {96},
	year = {2019}
}

@article{polyuga2012effort,
	author = {Rostyslav V. Polyuga and Arjan J. van der Schaft},
	doi = {10.1016/j.sysconle.2011.12.008},
	issn = {01676911},
	issue = {3},
	journal = {Systems and Control Letters},
	pages = {412-421},
	title = {Effort- and flow-constraint reduction methods for structure preserving model reduction of port-{H}amiltonian systems},
	volume = {61},
	year = {2012}
}

@article{ionescu2013families,
	author = {Tudor C. Ionescu and Alessandro Astolfi},
	doi = {10.1016/j.automatica.2013.05.006},
	issn = {00051098},
	issue = {8},
	journal = {Automatica},
	pages = {2424-2434},
	publisher = {Elsevier Ltd},
	title = {Families of moment matching based, structure preserving approximations for linear port-{H}amiltonian systems},
	volume = {49},
	year = {2013}
}

@article{Califano2021geometric,
	author = {Federico Califano and Ramy Rashad and Frederic P. Schuller and Stefano Stramigioli},
	doi = {10.1063/5.0048359},
	issn = {10897666},
	issue = {4},
	journal = {Physics of Fluids},
	month = {4},
	publisher = {American Institute of Physics Inc.},
	title = {Geometric and energy-aware decomposition of the {N}avier-{S}tokes equations: {A} port-{H}amiltonian approach},
	volume = {33},
	year = {2021}
}

@article{pagliantini2023gradient,
	author = {Cecilia Pagliantini and Federico Vismara},
	doi = {10.1137/22M1503890},
	issn = {10957197},
	issue = {5},
	journal = {SIAM Journal on Scientific Computing},
	month = {10},
	pages = {A2725-A2754},
	publisher = {Society for Industrial and Applied Mathematics Publications},
	title = {Gradient-preserving hyper-reduction of nonlinear dynamical systems via discrete empirical interpolation},
	volume = {45},
	year = {2023}
}

@inbook{antoulas2010interpolatory,
	author = {Athanasios C. Antoulas and Christopher A. Beattie and Serkan Gugercin},
	doi = {10.1007/978-1-4419-5757-3_1},
	booktitle = {Efficient modeling and control of large-scale systems},
	pages = {3-58},
	publisher = {Springer},
	title = {Interpolatory Model Reduction of Large-Scale Dynamical Systems},
	year = {2010}
}

@book{jacob2012linear,
	author = {Birgit Jacob and Hans Zwart},
	doi = {10.1007/978-3-0348-0399-1},
	publisher = {Springer Science \& Business Media},
	title = {Linear Port-{H}amiltonian Systems on Infinite-dimensional Spaces},
	year = {2012}
}

@article{buchfink2024model,
	author = {Patrick Buchfink and Silke Glas and Bernard Haasdonk and Benjamin Unger},
	doi = {10.1016/j.physd.2024.134299},
	issn = {01672789},
	journal = {Physica D: Nonlinear Phenomena},
	publisher = {Elsevier B.V.},
	title = {Model reduction on manifolds: {A} differential geometric framework},
	volume = {468},
	year = {2024}
}

@book{duindam2009modeling,
	author = {Vincent Duindam and Alessandro Macchelli and Stefano Stramigioli and Herman Bruyninckx},
	doi = {10.1007/978-3-642-03196-0},
	publisher = {Springer Science \& Business Media},
	title = {Modeling and control of complex physical systems the port-{H}amiltonian approach},
	year = {2009}
}

@phdthesis{annoni2016modeling,
	author = {Jennifer Annoni},
	school = {University of Minnesota},
	title = {Modeling for Wind Farm Control},
	year = {2016}
}

@inproceedings{ionescu2013moment,
	author = {Tudor C. Ionescu and Alessandro Astolfi},
	doi = {10.3182/20130904-3-RO-2023.00060},
	issue = {23},
	booktitle = {IFAC-Proceedings},
	pages = {395-399},
	title = {Moment matching for nonlinear port-{H}amiltonian and gradient systems},
	volume = {46},
	year = {2013}
}

@Inbook{Beattie2022,
author="Beattie, Christopher
and Gugercin, Serkan
and Mehrmann, Volker",
title="Structure-Preserving Interpolatory Model Reduction for Port-{H}amiltonian Differential-Algebraic Systems",
bookTitle="Realization and Model Reduction of Dynamical Systems: A Festschrift in Honor of the 70th Birthday of Thanos Antoulas",
year="2022",
publisher="Springer International Publishing",
address="Cham",
pages="235--254",
doi="10.1007/978-3-030-95157-3_13"
}

@article{falaize2016passive,
	author = {Antoine Falaize and Thomas Hélie},
	doi = {10.3390/app6100273},
	issn = {20763417},
	issue = {10},
	journal = {Applied Sciences (Switzerland)},
	month = {9},
	publisher = {MDPI AG},
	title = {Passive guaranteed simulation of analog audio circuits: {A} port-{H}amiltonian approach},
	volume = {6},
	year = {2016}
}

@article{wolf2010passivity,
	author = {Thomas Wolf and Boris Lohmann and Rudy Eid and Paul Kotyczka},
	doi = {10.3166/ejc.16.401-406},
	issn = {09473580},
	issue = {4},
	journal = {European Journal of Control},
	pages = {401-406},
	publisher = {European Control Association},
	title = {Passivity and structure preserving order reduction of linear port-{H}amiltonian systems using {K}rylov subspaces},
	volume = {16},
	year = {2010}
}

@article{morandin2023port,
	author = {Riccardo Morandin and Jonas Nicodemus and Benjamin Unger},
	doi = {10.1137/22M149329X},
	issn = {10957197},
	issue = {4},
	journal = {SIAM Journal on Scientific Computing},
	pages = {A1690-A1710},
	publisher = {Society for Industrial and Applied Mathematics Publications},
	title = {Port-{H}amiltonian dynamic mode decomposition},
	volume = {45},
	year = {2023}
}

@article{vanderschaft2014port,
	author = {Arjan J. van der Schaft and Dimitri Jeltsema},
	doi = {10.1561/2600000002},
	issn = {23256826},
	issue = {2-3},
	journal = {Foundations and Trends in Systems and Control},
	pages = {173-378},
	publisher = {Now Publishers Inc},
	title = {Port-{H}amiltonian systems theory: {A}n introductory overview},
	volume = {1},
	year = {2014}
}

@article{moore1981principal,
	author = {Bruce Moore},
	doi = {10.1109/TAC.1981.1102568},
	issue = {1},
	journal = {IEEE Transactions on Automatic Control},
	pages = {17-32},
	title = {Principal component analysis in linear systems: {C}ontrollability observability and model reduction},
	volume = {26},
	year = {1981}
}

@inproceedings{ohlberger2016reduced,
	author = {Mario Ohlberger and Stephan Rave},
	booktitle = {Proceedings of ALGORITMY},
	pages = {1-12},
	title = {Reduced basis methods: {S}uccess, limitations and future Challenges},
	year = {2016}
}

@article{sato2018riemannian,
	author = {Kazuhiro Sato},
	doi = {10.1016/j.automatica.2018.03.051},
	issn = {00051098},
	journal = {Automatica},
	month = {7},
	pages = {428-434},
	publisher = {Elsevier Ltd},
	title = {Riemannian optimal model reduction of linear port-{H}amiltonian systems},
	volume = {93},
	year = {2018}
}

@article{kawano2018structure,
	author = {Yu Kawano and Jacquelien M.A. Scherpen},
	doi = {10.1109/TAC.2018.2811787},
	issn = {15582523},
	issue = {12},
	journal = {IEEE Transactions on Automatic Control},
	pages = {4286-4293},
	publisher = {Institute of Electrical and Electronics Engineers Inc.},
	title = {Structure preserving truncation of nonlinear port-{H}amiltonian systems},
	volume = {63},
	year = {2018}
}

@article{chaturantabut2016structure,
	author = {Saifon Chaturantabut and Christopher A. Beattie and Serkan Gugercin},
	doi = {10.1137/15M1055085},
	issn = {10957197},
	issue = {5},
	journal = {SIAM Journal on Scientific Computing},
	pages = {B837-B865},
	publisher = {Society for Industrial and Applied Mathematics Publications},
	title = {Structure-preserving model reduction for nonlinear port-{H}amiltonian systems},
	volume = {38},
	year = {2016}
}

@article{schulze2023structure,
	author = {Philipp Schulze},
	doi = {10.3389/fams.2023.1160250},
	journal = {Frontiers in Applied Mathematics and Statistics},
	pages = {1160250},
	title = {Structure-preserving model reduction for port-{H}amiltonian systems based on separable nonlinear approximation ansatzes},
	volume = {9},
	year = {2023}
}

@article{gugercin2012structure,
	author = {Serkan Gugercin and Rostyslav V. Polyuga and Christopher Beattie and Arjan J. van der Schaft},
	doi = {10.1016/j.automatica.2012.05.052},
	issn = {00051098},
	issue = {9},
	journal = {Automatica},
	pages = {1963-1974},
	title = {Structure-preserving tangential interpolation for model reduction of port-{H}amiltonian systems},
	volume = {48},
	year = {2012}
}

@article{sirovich1987turbulence,
	author = {Lawrence Sirovich},
	doi = {10.1090/qam/910463},
	issue = {3},
	journal = {Quarterly of Applied mathematics},
	pages = {573-582},
	title = {Turbulence and the dynamics of coherent structures. II. {S}ymmetries and transformations},
	volume = {45},
	year = {1987}
}

@misc{Rashad2020,
	author = {Ramy Rashad and Federico Califano and Arjan J. van der Schaft and Stefano Stramigioli},
	doi = {10.1093/imamci/dnaa018},
	issn = {14716887},
	issue = {4},
	journal = {IMA Journal of Mathematical Control and Information},
	month = {12},
	pages = {1400-1422},
	publisher = {Oxford University Press},
	title = {Twenty years of distributed port-{H}amiltonian systems: {A} literature review},
	volume = {37},
	year = {2020}
}

@article{schwerdtner2023sobmor,
	author = {Paul Schwerdtner and Matthias Voigt},
	doi = {10.1137/20M1380235},
	issn = {10957197},
	issue = {2},
	journal = {SIAM Journal on Scientific Computing},
	pages = {A502-A529},
	publisher = {Society for Industrial and Applied Mathematics Publications},
	title = {{SOBMOR}: {S}tructured Optimization-Based Model Order Reduction},
	volume = {45},
	year = {2023}
}
	
\end{document}